\documentclass[a4paper,10pt]{article}

\usepackage{latexsym}
\usepackage{mathrsfs}
\usepackage{amsfonts,amsmath,amssymb}
\usepackage{indentfirst}
\usepackage{subeqnarray}
\usepackage{verbatim}
\usepackage[pdftex]{graphicx}
\usepackage{epstopdf}
\usepackage{stmaryrd}
\usepackage{multirow}
\usepackage{diagbox}
\usepackage{subcaption}
\usepackage{pdfpages}
\usepackage{algorithm}
\usepackage{algorithmic}
\usepackage[title]{appendix}
\usepackage{bm}

\newcommand{\bx}{\bm{x}}
\newcommand{\bn}{\bm{n}}
\newcommand{\br}{\bm{r}}
\newcommand{\bw}{\bm{w}}
\newcommand{\cB}{\mathcal{B}}
\newcommand{\cD}{\mathcal{D}}
\newcommand{\cI}{\mathcal{I}}
\newcommand{\cU}{\mathcal{U}}
\newcommand{\cL}{\mathcal{L}}
\newcommand{\balpha}{\bm{\alpha}}
\newcommand{\bOmega}{\bm{\Omega}}
\newcommand{\Real}{\mathbb{R}}

\newtheorem{theorem}{Theorem}[section]
\newtheorem{example}{Example}
\newtheorem{remark}[theorem]{Remark}
\numberwithin{equation}{section}

\newenvironment{proof}[1][Proof]{\textbf{#1.} }
{\ \rule{0.75em}{0.75em}\smallskip}

\usepackage[pdftex,unicode,colorlinks,linkcolor=blue,citecolor=red]{hyperref}

\begin{document}

\begin{center}
\Large\bf \textbf{Randomized Neural Networks for Integro-Differential Equations\\ with Application to Neutron Transport}
\end{center}

\begin{center}
    {\large Haoning Dang}\footnote{School of Mathematics and Statistics, Xi'an Jiaotong University, Xi'an, Shaanxi 710049, China. Email: {\tt haoningdang.xjtu@stu.xjtu.edu.cn}},\quad
    {\large Fei Wang}\footnote{School of Mathematics and Statistics \& State Key Laboratory of Multiphase Flow in Power Engineering, Xi'an Jiaotong University, Xi'an, Shaanxi 710049, China. The work of this author was partially supported by the National Natural Science Foundation of China (Grant No.\ 92470115). Email: {\tt feiwang.xjtu@xjtu.edu.cn}}, \quad
    {\large Yifan Chen}\footnote{School of Energy and Power Engineering, Xi'an Jiaotong University, Xi'an, Shaanxi 710049, China. Email: {\tt 2203221110@stu.xjtu.edu.cn}},\quad  
    {\large Zhouyu Liu}\footnote{School of Energy and Power Engineering, Xi'an Jiaotong University, Xi'an, Shaanxi 710049, China. The work of this author was partially supported by the National Natural Science Foundation of China (Grant No.\ 12375174) Email: {\tt zhouyuliu@xjtu.edu.cn}},\quad  
    {\large Dong Liu}\footnote{Nuclear Power Institute of China, State Key Laboratory of Advanced Nuclear Energy Technology, Chengdu, Sichuan 610213, China. The work of this author was partially supported by the National Natural Science Foundation of China (Grant No.\ 12575189). Email: {\tt liudong@npic.ac.cn}},\quad 
    {\large Hongchun Wu}\footnote{School of Energy and Power Engineering, Xi'an Jiaotong University, Xi'an, Shaanxi 710049, China. The work of this author was partially supported by the National Natural Science Foundation of China (Grant No.\ 12375175), and the Innovative Scientific Program of CNNC. Email: {\tt hongchun@xjtu.edu.cn}}
\end{center}

\medskip
\begin{quote}
{\bf Abstract. } Integro-differential equations arise in a wide range of applications, including transport, kinetic theory, radiative transfer, and multiphysics modeling, where nonlocal integral operators couple the solution across phase space. Such nonlocality often introduces dense coupling blocks in deterministic discretizations, leading to increased computational cost and memory usage, while physics-informed neural networks may suffer from expensive nonconvex training and sensitivity to hyperparameter choices. In this work, we present randomized neural networks (RaNNs) as a mesh-free collocation framework for linear integro-differential equations. Because the RaNN approximation is intrinsically dense through globally supported random features, the nonlocal integral operator does not introduce an additional loss of sparsity, while the approximate solution can still be represented with relatively few trainable degrees of freedom. By randomly fixing the hidden-layer parameters and solving only for the linear output weights, the training procedure reduces to a convex least-squares problem in the output coefficients, enabling stable and efficient optimization. As a representative application, we apply the proposed framework to the steady neutron transport equation, a high-dimensional linear integro-differential model featuring scattering integrals and diverse boundary conditions. Extensive numerical experiments demonstrate that, in the reported test settings, the RaNN approach achieves competitive accuracy while incurring substantially lower training cost than the selected neural and deterministic baselines, highlighting RaNNs as a robust and efficient alternative for the numerical simulation of nonlocal linear operators.

\end{quote}



\section{Introduction}

Integro-differential equations play a central role in modeling nonlocal interactions in many areas of science and engineering. Representative examples include kinetic transport and radiative transfer (\cite{DuderstadtHamilton1976NRA,Cercignani1988Boltzmann}), where a scattering or collision operator couples the solution across directions, energies, or other phase-space variables. From a computational standpoint, the principal difficulty is that the integral term is intrinsically nonlocal. Even when the differential operator admits sparse discretizations, the discretized integral operator typically introduces dense coupling in auxiliary variables (e.g., angles or groups), which can substantially increase computational and memory costs. Moreover, complex boundary conditions and multi-region material configurations further amplify these challenges.

Neural-network-based solvers offer an alternative route via mesh-free approximation. Physics-informed neural networks (PINNs, \cite{Raissi2019Physics}) embed governing equations and boundary conditions in a loss function and approximate solutions by training network parameters through nonlinear optimization. However, for integro-differential equations with strong nonlocal couplings, standard PINN training leads to a nonconvex optimization problem whose solution process may become trapped in local minima, often requiring extensive hyperparameter tuning and large computational budgets to reach satisfactory accuracy (\cite{Wang2021Gradient}). For applications of PINNs to integro-differential equations, one can refer to works such as \cite{Yuan2022APINN,Huhn2023FFPINNRTE}. Furthermore, many other neural-network-based approaches have been proposed for solving differential equations (\cite{E2018DRM,Sirignano2018DGM,Xu2020FiniteNeuron,Chen2023Friedrichs}), among the representative references listed here, direct applications to integro-differential equations appear to be relatively limited; see, for example, \cite{AlAradi2022DGM}.

In this work, we present Randomized Neural Networks (RaNNs, \cite{Huang2006ELM,Shang2023Randomized}) as a general methodology for solving linear integro-differential equations. The key idea is to randomize and fix the nonlinear parameters (hidden-layer weights and biases) and train only the linear output layer. This converts the learning problem into a linear least-squares problem, thereby (i) yielding a global minimizer of the empirical least-squares objective in the output coefficients, (ii) significantly reducing training cost relative to gradient-based training of fully trainable networks, and (iii) aligning naturally with linear operators in integro-differential models. Differential operators can be computed exactly by analytical differentiation of the fixed basis functions, while integral operators can be incorporated by standard numerical quadrature rules at collocation points, leading to a consistent and straightforward construction of the least-squares matrix. For heterogeneous or multi-material domains, the framework extends to local randomized neural networks (LRaNNs, \cite{Sun2022LRNNDG}) coupled through interface conditions (\cite{Li2023LRNN}). Because the RaNN discretization is intrinsically dense through globally supported random features, nonlocal integral terms do not cause an additional loss of sparsity as they do in many classical discretizations; meanwhile, competitive accuracy may still be achieved with relatively few trainable degrees of freedom.

To concretely illustrate the methodology, we apply RaNNs to the steady neutron transport equation (NTE), a high-dimensional linear integro-differential equation featuring a nonlocal scattering operator and practical boundary conditions. In this paper, the NTE serves as a case study to showcase how RaNNs handle nonlocal integral terms efficiently and how local networks can be coupled across interfaces in multi-material settings. In reactor physics applications, most existing neural approaches focus on diffusion models (\cite{Xie2021DLNeutronDiffusion,Wang2022cPINNNeutronDiffusion}) that avoid the integral scattering operator, while only a smaller number of studies directly target transport equations with integral terms (\cite{Liu2023NTE,Xie2024BDPINN,Liu2025IDNT}).

Classical deterministic solvers address such problems by combining suitable discretizations for the differential and integral parts, such as finite difference/element in space and discrete ordinates (\cite{CarlsonLathrop1965SN}) or spherical harmonics in direction. While effective, these approaches often suffer from a loss of sparsity in the coupling blocks due to the integral term. Although not fully dense, these blocks become significantly denser, leading to expensive linear algebra when high angular resolution or many groups are required. Stochastic approaches such as Monte Carlo (MC) provide flexibility for complicated geometries and physics but are limited by slow convergence when high accuracy is demanded (\cite{LuxKoblinger1991MC}). For background on neutron transport and representative methods, see \cite{LewisMiller1984CMNT,Kuridan2023NTE}.

The remainder of this paper is organized as follows. Section \ref{sec:RaNN} develops the RaNN framework for linear integro-differential equations, including least-squares formulation, numerical quadrature for integral terms, local coupling, and sketching acceleration. Section \ref{sec:NTE} specializes the framework to the neutron transport equation and its multigroup form as a demonstration model. Section \ref{sec:experiments} reports numerical experiments and Section \ref{sec:summary} concludes with a summary and outlook.

\section{Randomized Neural Networks for Integro-Differential Equations} \label{sec:RaNN}

Randomized Neural Networks (RaNNs) refer to neural network methods that incorporate randomness in a subset of parameters. The central mechanism is to randomize and fix certain nonlinear parameters (weights and biases in hidden neurons) and only train the linear output layer weights, thereby reducing computational cost and improving training efficiency. In this work, we focus on RaNNs with a single hidden layer; extensions to multilayer architectures can be found in \cite{Dang2024AGRNN}.

\begin{figure}[!htbp]
    \centering
    \includegraphics[width=0.35\textwidth]{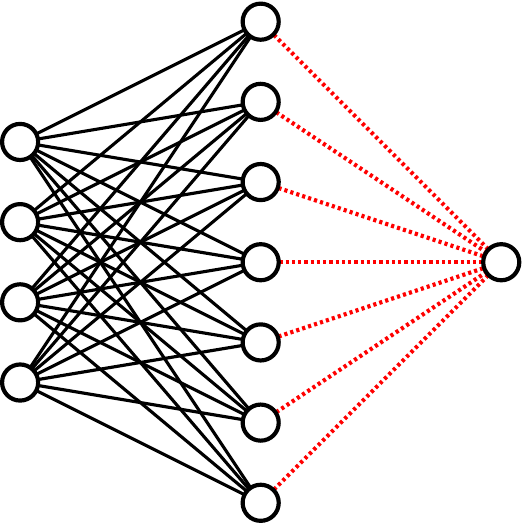}
    \caption{A schematic diagram of a randomized neural network (RaNN).}
    \label{fig:RaNN}
\end{figure}

As shown in Figure \ref{fig:RaNN}, the black solid lines represent randomly selected and fixed weights, while the red dashed lines denote trainable parameters. The fixed parameters are generated from a uniform distribution $\cU(-r,r)$ (see \cite{Dang2024AGRNN} for guidance on choosing $r$). Because the nonlinear parameters remain fixed, the network becomes a linear combination of predetermined nonlinear basis functions, and training reduces to finding the optimal linear coefficients.

\subsection{RaNN approximation as a linear basis expansion}
A single hidden layer RaNN takes the form
\begin{align} \label{Eq:basis_rewrite}
    \Psi_\rho(\bx)=\sum_{j=1}^m \alpha_j\psi_j(\bx)=\sum_{j=1}^m \alpha_j\rho(\bw_j\cdot\bx+b_j),
\end{align}
where $m$ is the number of neurons, $\rho$ is an activation function, $\bx\in\Real^d$ is the input in $d$ dimensions, $\bw_j\in\Real^d$ and $b_j\in\Real$ are fixed random parameters, and $\balpha=(\alpha_1,\cdots,\alpha_m)^\top$ are trainable linear coefficients. This viewpoint is particularly suitable for linear operators: once $\bw_j,b_j$ are fixed, applying a linear operator to $\Psi_\rho$ reduces to applying the operator to each basis function $\psi_j(\bx)$ and linearly combining the results.

\subsection{Least-squares formulation for linear differential and integral operators} \label{subsec:ls}

We consider a general linear integro-differential equation on a domain $D$,
\begin{align*}
    \cD\Psi(\bx)+\cI\Psi(\bx)=Q(\bx)\quad \text{in } D,
\end{align*}
supplemented with boundary constraints, denoted abstractly by
\begin{align*}
    \cB\Psi(\bx)=0\quad \text{on } \Gamma.
\end{align*}
Here $\cD$ is a linear differential operator and $\cI$ is a linear integral (nonlocal) operator. Using an $L^2$-type residual minimization, we define
\begin{align*}
    \cL(\Psi_{\rho})=\left\|\cD\Psi_{\rho}+\cI\Psi_{\rho}-Q\right\|_{0,D}^2+\left\|\omega^{1/2}\cB\Psi_{\rho}\right\|_{0,\Gamma}^2.
\end{align*}
Here, $\|f\|_{0,D}^2 := \int_D |f|^2\,\mathrm{d}\bx$ and
\begin{align*}
\left\|\omega^{1/2}g\right\|_{0,\Gamma}^2:=\int_\Gamma \omega(\bx)\,|g(\bx)|^2\,\mathrm{d}s
\end{align*}
denote the squared $L^2$ norms.

Let $\{\bx_i^I\}_{i=1}^{N^I} \subset D$ and $\{\bx_i^B\}_{i=1}^{N^B} \subset \Gamma$ be interior and boundary collocation points, respectively. Approximating the $L^2$ norms by collocation averages yields the empirical loss
\begin{align} \label{Eq:empirical-loss_rewrite}
    \widetilde{\cL}(\Psi_{\rho})&=\frac{|D|}{N^I}\sum_{i=1}^{N^I}\left(\cD\Psi_{\rho}(\bx_i^I)+\cI\Psi_{\rho}(\bx_i^I)-Q(\bx_i^I)\right)^2+\frac{|\Gamma|}{N^B}\sum_{i=1}^{N^B}\omega(\bx_i^B)\left(\cB\Psi_{\rho}(\bx_i^B)\right)^2 .
\end{align}

Substituting the RaNN ansatz \eqref{Eq:basis_rewrite} into \eqref{Eq:empirical-loss_rewrite} yields a quadratic function of $\balpha$:
\begin{align*}
    \bar{\cL}(\balpha)&=\frac{|D|}{N^I}\sum_{i=1}^{N^I}\left(\sum_{j=1}^m \alpha_j\cD\psi_j(\bx_i^I)+\sum_{j=1}^m \alpha_j\cI\psi_j(\bx_i^I)-Q(\bx_i^I)\right)^2+\frac{|\Gamma|}{N^B}\sum_{i=1}^{N^B}\omega(\bx_i^B)\left(\sum_{j=1}^m \alpha_j\cB\psi_j(\bx_i^B)\right)^2 .
\end{align*}
Equivalently, in matrix form,
\begin{align*}
    \bar{\cL}(\balpha)=\eta_1^2(A\balpha-F)^\top(A\balpha-F)+\eta_2^2(B_\omega\balpha)^\top(B_\omega\balpha),
\end{align*}
where
\begin{align*}
    A_{ij}=\cD\psi_j(\bx_i^I)+\cI\psi_j(\bx_i^I),\quad
    F_i=Q(\bx_i^I),\quad
    (B_\omega)_{ij}=\omega(\bx_i^B)^{1/2}\cB\psi_j(\bx_i^B),\quad
    \eta_1=\sqrt{\frac{|D|}{N^I}},\quad
    \eta_2=\sqrt{\frac{|\Gamma|}{N^B}}.
\end{align*}
The minimization of $\bar{\cL}(\balpha)$ is equivalent to the weighted least-squares problem
\begin{align} \label{Eq:ls-prob_rewrite}
\min_{\balpha}\left\|\left[\begin{matrix}\eta_1A\\\eta_2 B_\omega\end{matrix}\right]\balpha-\left[\begin{matrix}\eta_1F\\\bm{0}\end{matrix}\right]\right\|_2^2 .
\end{align}

\paragraph{Why this is advantageous for integro-differential equations.}
In many deterministic discretizations, the integral operator induces strong coupling among auxiliary degrees of freedom (e.g., angular directions or energy groups). As the resolution is increased, this coupling typically enlarges the effective system size and can lead to substantial computational and memory costs. In the RaNN formulation, the approximate solution is represented in a randomized neural function space with $m$ basis functions, and the unknowns are only the output coefficients $\balpha\in\mathbb{R}^m$. Because these basis functions are globally supported, the resulting least-squares system \eqref{Eq:ls-prob_rewrite} is intrinsically dense. Consequently, the nonlocal integral operator does not introduce an additional loss of sparsity or a qualitative change in the algebraic structure of the discretization; its main effect is on the operator-evaluation stage during matrix assembly. At the same time, competitive accuracy can often be achieved with a relatively small number of trainable degrees of freedom. The main advantage of the RaNN formulation is therefore twofold: robustness with respect to dense nonlocal couplings, and efficient solution representation with comparatively few unknown coefficients. Moreover, for fixed random features, the training objective is a convex quadratic function of $\balpha$, and the coefficients can be computed efficiently by standard least-squares solvers without the local-minimum issues associated with gradient-based training of fully trainable networks.

\begin{remark}
    For nonlinear problems, the governing equations can be linearized using iterative methods such as Picard iteration or Newton iteration.
\end{remark}

\subsection{Operator evaluation: exact differentiation and quadrature for integral terms}
A key practical step is the evaluation of $\cD\psi_j$ and $\cI\psi_j$ for each basis function. The differential component $\cD\psi_j$ can be computed exactly via analytical differentiation. The integral component $\cI\psi_j$ is generally unavailable in closed form and is therefore approximated by numerical quadrature.

Let $\{\xi_k\}_{k=1}^K$ be quadrature points for the integral variable(s) and $\{\beta_k\}_{k=1}^K$ the corresponding weights. Then, we approximate
\begin{align*}
    \cI\psi_j(\bx)=\int \mathcal{K}(\bx,\xi)\psi_j(\xi)\,\mathrm{d}\xi\approx \sum_{k=1}^K \beta_k\, \mathcal{K}(\bx,\xi_k)\,\psi_j(\xi_k).
\end{align*}
In the neutron-transport experiments, the scattering integral is evaluated by quadrature over the angular variables at each spatial location. In the numerical experiments reported below, the integral terms are approximated using the trapezoidal rule.

\subsection{Local RaNNs and interface coupling}
Many applications involve heterogeneous media or multiple subdomains. Inspired by local randomized neural network ideas for interface problems (\cite{Li2023LRNN}), we extend the present formulation to multi-material transport settings by coupling subdomain networks through weighted interface residuals. Consider two subdomains $D_1$ and $D_2$ with interface $\Gamma_I=\bar{D}_1\cap\bar{D}_2$, and define $\Psi_\rho^1$ and $\Psi_\rho^2$ on $D_1$ and $D_2$, respectively. Accordingly,
\begin{align*}
    \Psi_\rho(\bx)=\begin{cases}
        \Psi_\rho^1(\bx),\quad\bx\in D_1,\\\Psi_\rho^2(\bx),\quad\bx\in D_2,
    \end{cases}
\end{align*}
where $\Psi_\rho^i(\bx)$ has $m_i$ neurons and the corresponding uniform distribution parameter is $r_i$ ($i=1,2$). To impose continuity of the angular flux across the interface, we sample $\{\bx_i^F\}_{i=1}^{N^F}\subset \Gamma_I$ and add the weighted interface penalty term. Then, the empirical loss \eqref{Eq:empirical-loss_rewrite} can be written as 
\begin{align*}
    \widetilde{\cL}(\Psi_{\rho})=\frac{|D|}{N^I}\sum_{i=1}^{N^I}\left(\cD\Psi_{\rho}(\bx_i^I)+\cI\Psi_{\rho}(\bx_i^I)-Q(\bx_i^I)\right)^2+\frac{|\Gamma|}{N^B}\sum_{i=1}^{N^B}\omega(\bx_i^B)\left(\cB\Psi_{\rho}(\bx_i^B)\right)^2+\frac{|\Gamma_I|}{N^F}\sum_{i=1}^{N^F}\omega_I(\bx_i^F)\big(\Psi_\rho^1(\bx_i^F)-\Psi_\rho^2(\bx_i^F)\big)^2 .
\end{align*}
This interface term adds residual rows to the augmented weighted least-squares system, coupling the output coefficients across subdomains in a straightforward manner.

\subsection{Comparison with classical deterministic discretizations}

For many integro-differential equations, the integral term serves as a nonlocal source involving integrals of the solution over auxiliary variables. While the differential operator may yield a sparse discretization, the integral operator generally induces dense couplings among local degrees of freedom (e.g., angular or group variables in transport problems), leading to dense blocks in the system matrix and increased computational cost.

For the RaNN approach, the primary computational task is the construction of the matrices $A$, $B$ and vector $F$ in \eqref{Eq:ls-prob_rewrite}. Since derivatives are computed exactly and integrals are computed via quadrature, the assembly consists primarily of evaluating basis functions and performing matrix multiplications. The least-squares matrix is dense due to global basis functions; thus, the nonlocal coupling does not cause an additional qualitative degradation of sparsity, and the method remains practically viable for the studied problems, although the assembly and least-squares costs still grow with the collocation size, quadrature complexity, and multigroup coupling.

\subsection{Universal approximation and training stability}
In many integro-differential models, the solution is a continuous function on a bounded domain (or a set of such functions in multigroup settings). RaNNs possess universal approximation properties in Sobolev norms (\cite{Dang2024AGRNN}), consistent with trainable neural networks. 

In terms of training, conventional training requires solving a nonlinear and nonconvex optimization problem, often with many local minima and saddle points, which may fail to deliver satisfactory accuracy without extensive tuning. In contrast, RaNN training solves a convex quadratic problem in $\balpha$, thereby guaranteeing a global minimizer of the empirical loss and improving training stability. If the resulting matrix is rank deficient, we can further add Tikhonov regularization to ensure the uniqueness of the solution.

\subsection{Sketching acceleration for large least-squares systems} \label{sec:sketching}

Sketching is a randomized technique for accelerating the solution of large least-squares problems. We briefly describe its use in the present setting; theoretical analysis can be found in \cite{Woolfe2008random, Epperly2024random,Shan2025LS}. Denote the augmented least-squares problem \eqref{Eq:ls-prob_rewrite} in the compact form
\begin{align*}
    \min_{\balpha}\|\tilde{A}\balpha-\tilde{F}\|_2^2,
\end{align*}
where $\tilde{A}\in\mathbb{R}^{(N^I+N^B)\times m}$ and $\tilde{F}\in\mathbb{R}^{(N^I+N^B)}$. 
Sketching introduces a random embedding matrix $S\in\mathbb{R}^{N^S\times (N^I+N^B)}$ to form the compressed problem
\begin{align*}
    \min_{\balpha}\|S\tilde{A}\balpha-S\tilde{F}\|_2^2,
\end{align*}
where $N^S\ll N^I+N^B$. 
In this work, we take $N^S = m d_S$ and construct $S$ as a sparse matrix with exactly $n_S$ nonzero entries per row, whose values are independently chosen as $\pm\sqrt{1/n_S}$.

\begin{remark}
The sketching matrix $S$ is constructed by randomly selecting $n_S$ rows of $\tilde{A}$ for each sketched row and forming their linear combinations with coefficients $\pm\sqrt{1/n_S}$. As a result, the sketched system contains $m d_S$ rows and exhibits a highly sparse structure, leading to low computational and memory costs when forming $S\tilde{A}$ and $S\tilde{F}$. Throughout all numerical experiments, we take $d_S = 2$ and $n_S = 8$.
\end{remark}

\section{Steady Neutron Transport Equation} \label{sec:NTE}
To demonstrate the above RaNN framework on a representative high-dimensional nonlocal model, we consider the steady neutron transport equation (NTE), a linear integro-differential equation that describes neutron angular flux in reactor physics.

\subsection{Steady transport equation and boundary conditions}

Consider the steady-state neutron transport equation
\begin{align} \label{Eq:NTE}
    \bOmega\cdot\nabla_{\br}\Psi(\br,\bOmega,E)+\Sigma_t(\br,\bOmega,E)\Psi(\br,\bOmega,E)=\int_0^\infty\int_{4\pi}\Sigma_s(\br;\bOmega'\to\bOmega,E'\to E)\Psi(\br,\bOmega',E')\,\mathrm{d}\bOmega'\mathrm{d}E'+Q(\br,\bOmega,E)\text{ in }D,
\end{align}
where $\Psi$ denotes the unknown neutron angular flux. The phase-space domain is $D = D_{\br} \times D_{\bOmega} \times D_E$, 
with $\br$ the spatial variable, $\bOmega$ the transport direction, and $E$ the neutron energy. The functions $\Sigma_t$ and $\Sigma_s$ are the macroscopic total and scattering cross sections, respectively, and $Q$ is a prescribed source term. In Cartesian coordinates, $\br = (x,y,z)$ and the direction vector $\bOmega$ is parameterized as
\begin{align*}
    \bOmega=(\sqrt{1-\mu^2}\cos\varphi,\; \sqrt{1-\mu^2}\sin\varphi,\; \mu),
\end{align*}
with $\mu\in[-1,1]$ and $\varphi\in[0,2\pi)$. Accordingly, integration over the unit sphere is given by
\begin{align*}
    \int_{4\pi}\mathrm{d}\bOmega=\int_{-1}^{1}\!\int_{0}^{2\pi}\mathrm{d}\varphi\,\mathrm{d}\mu .
\end{align*}

Define the inflow boundary $\Gamma^-=\{(\br,\bOmega,E): \br\in\partial D_{\br},\, \bn_{\br}\cdot\bOmega<0\}$, where $\bn_{\br}$ is the unit outward normal vector. Two typical boundary conditions are:
\begin{itemize}
\item Vacuum boundary condition,
\begin{align} \label{Eq:bdcond1}
    \Psi(\br,\bOmega,E)=0 \text{ on }\Gamma^-.
\end{align}
\item Reflecting boundary condition,
\begin{align} \label{Eq:bdcond2}
    \Psi(\br,\bOmega,E)=\Psi(\br,\bOmega_r,E)\text{ on }\Gamma^-,
\end{align}
where $\bOmega_r$ is the reflecting angle satisfying $\bn_{\br}\cdot\bOmega_r=-\bn_{\br}\cdot\bOmega$ and $(\bOmega\times\bOmega_r)\cdot\bn_{\br}=0$.
\end{itemize}

\subsection{Multigroup approximation}
In neutron transport problems, the neutron energy can vary continuously from tens of MeV down to 0 eV. The multigroup method is the most commonly used and straightforward discretization approach. In this method, the energy range of the neutron flux is divided into $G$ intervals $(E_G,E_{G-1}),\cdots,(E_g,E_{g-1}),\cdots, (E_1,E_0)$. Each such energy interval $\Delta E_g=E_{g-1}-E_g$ is referred to as an energy group. The group index $g$ is typically numbered in descending order from high to low energy. Integrating the neutron transport equation \eqref{Eq:NTE} over each energy interval $\Delta E_g$ yields the following system of equations:
\begin{align} \label{Eq:multigroup}
    \bOmega\cdot\nabla\Psi_g(\br,\bOmega)+\Sigma_{t,g}(\br,\bOmega)\Psi_g(\br,\bOmega)=\sum_{g'=1}^{G}\int_{4\pi} \Sigma_{s,g'\to g}(\br;\bOmega'\to\bOmega)\Psi_{g'}(\br,\bOmega')\,\mathrm{d}\bOmega'+Q_g(\br,\bOmega),
    \quad g=1,\cdots,G,
\end{align}
where
\begin{align*}
    \Psi_g(\br,\bOmega)=\int_{E_g}^{E_{g-1}} \Psi(\br,\bOmega,E)\,\mathrm{d}E
\end{align*}
is the group angular flux, defined as the integral of the angular flux over energy group $g$, and
\begin{align*}
    Q_g(\br,\bOmega)=\int_{E_g}^{E_{g-1}}Q(\br,\bOmega,E)\,\mathrm{d}E.
\end{align*}
Here, we are not concerned with how the group constants $\Sigma_{t,g}(\br,\bOmega)$ and $\Sigma_{s,g'\to g}(\br;\bOmega'\to\bOmega)$ are calculated, but rather treat them as known functions. Note that we will omit the subscript $g$ in the one-group case. 

In practice, we are primarily interested in the neutron scalar flux:
\begin{align} \label{Eq:flux}
    \Phi_g(\br)=\int_{4\pi}\Psi_g(\br,\bOmega)\,\mathrm{d}\bOmega.
\end{align}
Accordingly, all numerical results reported below are compared in terms of the neutron scalar flux.

\subsection{Specialization of the RaNN least-squares system to the NTE}
For the one-group form of \eqref{Eq:multigroup}, we write
\begin{align*}
    \cD\Psi(\br,\bOmega)+\cI\Psi(\br,\bOmega)=Q,
\end{align*}
with
\begin{align*}
    \cD\Psi(\br,\bOmega)=\bOmega\cdot\nabla\Psi(\br,\bOmega)+\Sigma_t(\br,\bOmega)\Psi(\br,\bOmega),
    \qquad
    \cI\Psi(\br,\bOmega)=-\int_{4\pi} \Sigma_s(\br;\bOmega'\to\bOmega)\Psi(\br,\bOmega')\,\mathrm{d}\bOmega',
\end{align*}
and denote the boundary condition \eqref{Eq:bdcond1} or \eqref{Eq:bdcond2} by $\cB\Psi(\br,\bOmega)=0$. Writing $(\br,\bOmega)$ as $\bx$, the empirical loss \eqref{Eq:empirical-loss_rewrite} and least-squares system \eqref{Eq:ls-prob_rewrite} apply directly.

For the integral operator, let $\{\bOmega_k\}_{k=1}^K$ and $\{\beta_k\}_{k=1}^K$ be quadrature points and weights. Then
\begin{align*}
    \cI\psi_j(\br,\bOmega)
    =-\int_{4\pi} \Sigma_s(\br;\bOmega'\to\bOmega)\psi_j(\br,\bOmega')\,\mathrm{d}\bOmega'
    \approx-\sum_{k=1}^K\beta_k\Sigma_s(\br;\bOmega_k\to\bOmega)\psi_j(\br,\bOmega_k).
\end{align*}
In this work, the quadrature points are taken to be the collocation points, and the trapezoidal rule is used throughout.

\subsection{Weighted loss for neutron transport problems}
\label{subsec:weighted-loss-nte}

For neutron transport equations, the boundary and interface residuals should be weighted by the transport trace factor determined by the normal component of the angular direction. More precisely, for a vacuum boundary point $(\br,\bOmega)\in \partial D_{\br}\times D_{\bOmega}$ with outward unit normal $\bn_{\br}$, we take
\begin{align*}
    \omega(\br,\bOmega)=|\bn_{\br}\cdot\bOmega|.
\end{align*}
For an interior interface point $(\br,\bOmega)\in \Gamma_I$ with unit normal $\bn_I$, we analogously define
\begin{align*}
    \omega_I(\br,\bOmega)=|\bn_I\cdot\bOmega|.
\end{align*}
These weights are consistent with the natural trace norms for first-order transport operators and measure the strength of the flux crossing the boundary or interface.

Accordingly, for the vacuum boundary condition, we take
\begin{align*}
\cB\Psi(\br,\bOmega)=\Psi(\br,\bOmega).
\end{align*}
For the reflecting boundary condition, we take
\begin{align*}
\cB\Psi_\rho(\br,\bOmega)=\Psi_\rho(\br,\bOmega)-\Psi_\rho(\br,\bOmega_r).
\end{align*}
The weighted least-squares loss for the one-group neutron transport equation is then written as
\begin{align*}
    \cL(\Psi_\rho)=\|\cD\Psi_\rho+\cI\Psi_\rho-Q\|_{0,D}^2+\|\omega^{1/2}\cB\Psi_\rho\|_{0,\Gamma^-}^2,
\end{align*}
where
\begin{align*}
    \left\|\omega^{1/2}\cB\Psi_\rho\right\|_{0,\Gamma^-}^2:=\int_{\Gamma^-}|\bn_{\br}\cdot\bOmega||\cB\Psi_\rho(\br,\bOmega)|^2\,\mathrm{d}\sigma(\br)\,\mathrm{d}\bOmega.
\end{align*}

If local RaNNs are used on two subdomains $D_1$ and $D_2$, then the interface continuity can be imposed through the weighted penalty
\begin{align*}
    \left\|\omega_I^{1/2}(\Psi_\rho^1-\Psi_\rho^2)\right\|_{0,\Gamma_I}^2:=\int_{\Gamma_I}|\bn_I\cdot\bOmega||\Psi_\rho^1(\br,\bOmega)-\Psi_\rho^2(\br,\bOmega)|^2\,\mathrm{d}\sigma(\br)\,\mathrm{d}\bOmega .
\end{align*}
Hence, the corresponding weighted loss becomes
\begin{align*}
    \cL(\Psi_\rho)=\|\cD\Psi_\rho+\cI\Psi_\rho-Q\|_{0,D}^2+\|\omega^{1/2}\cB\Psi_\rho\|_{0,\Gamma^-}^2+\left\|\omega_I^{1/2}(\Psi_\rho^1-\Psi_\rho^2)\right\|_{0,\Gamma_I}^2.
\end{align*}

For the angular parametrization
\begin{align*}
    \bOmega=\big(\sqrt{1-\mu^2}\cos\varphi,\ \sqrt{1-\mu^2}\sin\varphi,\ \mu\big),
\end{align*}
the weight can be written explicitly once the normal direction is fixed. For example, in the one-dimensional slab geometry, $\omega=|\mu|$, for a vertical interface or boundary with normal $\bn=(\pm1,0,0)$, one has
\begin{align*}
    |\bn\cdot\bOmega|=\sqrt{1-\mu^2}\,|\cos\varphi|.
\end{align*}
For a horizontal boundary with normal $\bn=(0,\pm1,0)$, one has
\begin{align*}
    |\bn\cdot\bOmega|=\sqrt{1-\mu^2}\,|\sin\varphi|.
\end{align*}
Therefore, in each neutron transport example below, the weight is determined by the corresponding transport trace factor $|\bn\cdot\bOmega|$ on the physical boundary or interface.

\subsection{Relation to the transport graph norm and inflow-boundary norm}
\label{subsec:graphnorm}

To clarify the analytical meaning of the least-squares residual for transport problems, we briefly relate it to the natural graph norm of the streaming--absorption operator and the associated weighted inflow-boundary norm. For first-order transport equations on bounded domains, weighted trace spaces on the inflow and outflow boundaries, together with the corresponding Green identity, are standard (\cite{Dahmen2012APG,Tervo2021Regularity}). Related boundary functionals have also been used in least-squares formulations of neutron transport to establish stability and error estimates (\cite{Manteuffel2000Boundary}).

For the one-group transport problem, let
\begin{align*}
D := D_{\br}\times D_{\bOmega},
\qquad
\Gamma^\pm := \{(\br,\bOmega)\in \partial D_{\br}\times D_{\bOmega}:\ \pm\,\bn_{\br}\cdot\bOmega>0\}.
\end{align*}
We define the weighted outflow/inflow boundary norms by
\begin{align*}
    \|\omega^{1/2}g\|_{0,\Gamma^\pm}^2:=\int_{\Gamma^\pm}|\bn_{\br}\cdot\bOmega||g(\br,\bOmega)|^2\,\,\mathrm{d}\sigma(\br)\,\mathrm{d}\bOmega,
\end{align*}
and the transport graph norm by
\begin{align*}
    \|\Psi\|_{\mathrm{gr}}^2:=\|\Psi\|_{0,D}^2+\|\bOmega\cdot\nabla_{\br}\Psi\|_{0,D}^2.
\end{align*}
We also introduce the transport graph space
\begin{align*}
    W:=\left\{\Psi\in L^2(D):\bOmega\cdot\nabla_{\br}\Psi\in L^2(D)\right\}.
\end{align*}
Under the geometric assumptions stated below, standard transport trace theory implies that both traces
\begin{align*}
    \gamma^\pm : W \to L^2(\Gamma^\pm, |\bn_{\br}\cdot\bOmega|\,\mathrm{d}\sigma\,\mathrm{d}\bOmega)
\end{align*}
are well defined and continuous (\cite{Dahmen2012APG,Tervo2021Regularity}).

We now consider the streaming--absorption operator
\begin{align*}
    \cD \Psi := \bOmega\cdot\nabla_{\br}\Psi + \Sigma_t(\br,\bOmega)\,\Psi.
\end{align*}
Assume throughout this subsection that
\begin{align}\label{Eq:sigma-bounds}
    0<\sigma_* \le \Sigma_t(\br,\bOmega)\le \sigma^*<\infty\qquad \text{for a.e. }(\br,\bOmega)\in D.
\end{align}

\begin{theorem}[Graph-norm control by the transport residual]\label{thm:graphnorm}
Assume that $D_{\br}\subset\Real^d$ is a bounded convex domain with Lipschitz boundary, and that \eqref{Eq:sigma-bounds} holds. Define
\begin{align*}
    W_0^- := \{\Psi\in W:\ \gamma^-\Psi=0\}.
\end{align*}
Then, for every $\Psi\in W_0^-$, the following estimate holds:
\begin{align}\label{Eq:graph-control-by-D}
    \|\Psi\|_{\mathrm{gr}}\le C_{\mathrm{gr}}\,\|\cD \Psi\|_{0,D},\qquad C_{\mathrm{gr}}:=\left[\sigma_*^{-2}+\left(1+\frac{\sigma^*}{\sigma_*}\right)^2\right]^{1/2}.
\end{align}
Conversely,
\begin{align}\label{Eq:D-control-by-graph}
    \|\cD \Psi\|_{0,D}\le(1+\sigma^*)\,\|\Psi\|_{\mathrm{gr}}.
\end{align}
Hence, on $W_0^-$, the residual norm $\|\cD \Psi\|_{0,D}$ and the transport graph norm $\|\Psi\|_{\mathrm{gr}}$ are equivalent.
\end{theorem}

\begin{proof}
We first prove the estimate for smooth functions in $W_0^-$. By the transport Green identity (\cite{Dahmen2012APG,Tervo2021Regularity}),
\begin{align*}
    \int_{D}(\bOmega\cdot\nabla_{\br}\Psi)\,\Psi\,\mathrm{d}\br\,\mathrm{d}\bOmega=\frac12\int_{\partial D_{\br}\times D_{\bOmega}}(\bn_{\br}\cdot\bOmega)\,|\Psi|^2\,\mathrm{d}\sigma(\br)\,\mathrm{d}\bOmega.
\end{align*}
Splitting the boundary integral into inflow and outflow parts yields
\begin{align*}
    (\bOmega\cdot\nabla_{\br}\Psi,\Psi)_{L^2(D)}=\frac12\|\omega^{1/2}\gamma^+\Psi\|_{0,\Gamma^+}^2-\frac12\|\omega^{1/2}\gamma^-\Psi\|_{0,\Gamma^-}^2.
\end{align*}
Since $\gamma^-\Psi=0$, we obtain
\begin{align}
    (\cD \Psi,\Psi)_{L^2(D)}=(\bOmega\cdot\nabla_{\br}\Psi,\Psi)_{L^2(D)}+(\Sigma_t \Psi,\Psi)_{L^2(D)}=\frac12\|\omega^{1/2}\gamma^+\Psi\|_{0,\Gamma^+}^2+(\Sigma_t \Psi,\Psi)_{L^2(D)}\ge\sigma_*\,\|\Psi\|_{0,D}^2.
    \label{Eq:D-coercive}
\end{align}
By Cauchy--Schwarz,
\begin{align*}
    (\cD \Psi,\Psi)_{L^2(D)}\le\|\cD \Psi\|_{0,D}\,\|\Psi\|_{0,D},
\end{align*}
which together with \eqref{Eq:D-coercive} implies
\begin{align}\label{Eq:basic-L2-control}
    \|\Psi\|_{0,D}&\le\sigma_*^{-1}\,\|\cD \Psi\|_{0,D}.
\end{align}

Next, using
\begin{align*}
\bOmega\cdot\nabla_{\br}\Psi
=
\cD \Psi - \Sigma_t \Psi,
\end{align*}
we obtain
\begin{align}\label{Eq:basic-streaming-control}
    \|\bOmega\cdot\nabla_{\br}\Psi\|_{0,D}\le\|\cD \Psi\|_{0,D}+\|\Sigma_t \Psi\|_{0,D}\le\|\cD \Psi\|_{0,D}+\sigma^*\,\|\Psi\|_{0,D}\le\left(1+\frac{\sigma^*}{\sigma_*}\right)\|\cD \Psi\|_{0,D}.
\end{align}
Combining \eqref{Eq:basic-L2-control} and \eqref{Eq:basic-streaming-control} yields \eqref{Eq:graph-control-by-D}.

For the converse estimate, we write
\begin{align*}
    \|\cD \Psi\|_{0,D}\le\|\bOmega\cdot\nabla_{\br}\Psi\|_{0,D}+\|\Sigma_t \Psi\|_{0,D}\le\|\bOmega\cdot\nabla_{\br}\Psi\|_{0,D}+\sigma^*\,\|\Psi\|_{0,D}\le(1+\sigma^*)\,\|\Psi\|_{\mathrm{gr}},
\end{align*}
which is \eqref{Eq:D-control-by-graph}. The extension from smooth functions to all $\Psi\in W_0^-$ follows from the standard density argument together with the continuity of the weighted trace operators.
\end{proof}

The preceding theorem shows that, on the homogeneous inflow space, controlling the $L^2$ residual of the streaming--absorption operator is equivalent to controlling the natural transport graph norm. In particular, under the vacuum boundary condition, the inflow datum satisfies the homogeneous inflow. 

\begin{remark}[Residual control with inflow mismatch]
Nonhomogeneous inflow data can likewise be treated by means of a lifting argument. Assume the hypotheses of Theorem \ref{thm:graphnorm}, and let $\Psi^*\in W$ be the solution of
\begin{align*}
    \cD \Psi^* = f \quad \text{in } D,\qquad\gamma^-\Psi^* = g \quad \text{on } \Gamma^-,
\end{align*}
where $f\in L^2(D)$ and $g\in L^2(\Gamma^-,|\bn_{\br}\!\cdot\!\bOmega|\,\mathrm d\sigma\,\mathrm d\bOmega)$. If the inflow trace admits a bounded lifting into $W$, then for any $\Psi\in W$,
\begin{align*}
    \|\Psi^*-\Psi\|_{\mathrm{gr}}\le C_{\mathrm{res}}\Big(\|\cD\Psi-f\|_{0,D}+\|\omega^{1/2}(\gamma^-\Psi-g)\|_{0,\Gamma^-}\Big),
\end{align*}
where $C_{\mathrm{res}}$ depends only on $\sigma_*$, $\sigma^*$, and the lifting bound.
\end{remark}

\begin{remark}[Scope for reflecting boundaries]
    Theorem \ref{thm:graphnorm} is stated for vacuum boundary. For reflecting boundary, the same weighted trace norm can be applied to the boundary residual $\cB\Psi_\rho$, but the corresponding well-posedness argument requires the reflecting boundary operator to be incorporated explicitly into the trace formulation. We do not pursue that extension here.
\end{remark}

\begin{remark}[Role of the scattering operator]
    The graph norm above is determined by the streaming--absorption operator $\cD$. For the full transport operator $\cD+\cI$, the scattering term $\cI$ acts as a bounded perturbation under standard boundedness assumptions on the scattering kernel. Consequently, the full residual may still serve as a practically meaningful objective. However, the rigorous graph-norm equivalence established in this paper applies only to the streaming--absorption part $\cD$. If, in addition, $\cI$ is sufficiently small relative to the absorption in the sense that
\begin{align*}
    \|\cI \Psi\|_{0,D}\le C_I\|\Psi\|_{0,D}\qquad\text{with } C_I<\sigma_*,
\end{align*}
then the above argument extends verbatim and yields a genuine graph-norm estimate with $\cD$ replaced by $\cD+\cI$. This smallness assumption is not imposed in the numerical experiments below, it is stated only to illustrate a regime in which the full operator inherits the same coercivity mechanism. For general scattering-dominated regimes, it is therefore most accurate to interpret the residual of $\cD+\cI$ as a bounded perturbation of the graph-norm-controlling quantity associated with $\cD$.

If the cross sections are constant, then the boundedness of the scattering operator $\cI$ on $L^2(D)$ can often be verified explicitly. However, constancy of the coefficients alone does not imply the smallness condition $C_I<\sigma_*$.  In the one-group isotropic case, one has $C_I=\Sigma_s$, so that $C_I<\sigma_*$ is equivalent to $\Sigma_s<\Sigma_t$, i.e., to the presence of strictly positive absorption. 
\end{remark}

\section{Numerical Experiments} \label{sec:experiments}
In this section, we evaluate the RaNN framework on the neutron transport problems described above. All experiments are conducted in MATLAB. To ensure reproducibility, the random number generator is fixed using ``rng(1)''. The least-squares problems are solved via the QR method using the ``linsolve'' function with option ``opts.RECT = true''. The sketching technique is used in Example \ref{ex3}--\ref{ex5}; all other examples solve the original least-squares problem. All numerical results are reported for the neutron scalar flux \eqref{Eq:flux}.

Examples \ref{ex1}--\ref{ex2} were run on a desktop (Intel Core i7-13700KF, 32 GB DDR5), while other examples were run on a server (dual Intel Xeon Gold 6240, 192 GB DDR4).

The activation function is chosen as the Gaussian
\begin{align*}
\rho(x)=\exp(-x^2/2).
\end{align*}
Accuracy is measured by the relative $\ell^2$ error of the scalar flux on a spatial test grid,
\begin{align*}
    e_{\ell^2}(\Phi_\rho)
    =
    \sqrt{
    \frac{\sum_{i=1}^N\left(\Phi_\rho(\bx_i)-\Phi(\bx_i)\right)^2}
         {\sum_{i=1}^N\left(\Phi(\bx_i)\right)^2}
    },
\end{align*}
where $\Phi$ denotes the reference neutron scalar flux, and $\bx_i$ denotes a test point at which a reference value is available.

Since analytical solutions are generally unavailable, we use NECP-MCX (hereafter abbreviated as MCX) to generate reference solutions for all examples except Example \ref{ex1}. MCX (\cite{He2021MCX}) is based on a hybrid Monte Carlo--deterministic methodology, in which deterministic calculations are used to construct mesh-based weight-window and source-biasing parameters that guide the Monte Carlo simulation and reduce variance.

For the pin-cell problems in Examples \ref{ex3}--\ref{ex5}, the MCX input parameters and the corresponding simulation statistics are reported in Table \ref{table:MCX}. A uniform $50\times 50$ mesh is used throughout, and we report the minimum, maximum, and average relative standard deviations (RSD) over all grid cells.

\begin{table}[!htbp]
\centering
\begin{tabular}{|c|c|c|c|c|c|c|c|}
\hline
\multicolumn{2}{|c|}{} & batch & particles & RSD$_\text{min}$ & RSD$_\text{max}$ & RSD$_\text{average}$ \\ \hline\hline
\multirow{3}{*}{Example \ref{ex3}} & Case 1 & 1000 & 100000 & 7.2572e-04 & 1.6919e-03 & 9.9630e-04 \\ \cline{2-7} 
    & Case 2 & 10000 & 500000 & 5.3084e-04 & 1.0547e-03 & 6.6250e-04 \\ \cline{2-7} 
    & Case 3 & 1000 & 200000 & 6.1908e-04 & 1.2144e-03 & 7.5330e-04 \\ \hline\hline
\multirow{3}{*}{Example \ref{ex4}} & Case 1 & 1000 & 20000 & 6.0193e-04 & 9.7435e-04 & 6.7397e-04 \\ \cline{2-7} 
    & Case 2 & 1000 & 20000 & 2.7252e-03 & 4.1608e-03 & 3.0236e-03 \\ \cline{2-7} 
    & Case 3 & 1000 & 20000 & 6.1269e-04 & 9.5861e-04 & 6.8497e-04 \\ \hline\hline
    & Group 1 & \multirow{7}{*}{2000} & \multirow{7}{*}{2000000} & 6.7725e-05 & 1.0417e-04 & 7.3931e-05 \\ \cline{2-2} \cline{5-7}
    & Group 2 &  &  & 1.4741e-04 & 2.1600e-04 & 1.5878e-04 \\ \cline{2-2} \cline{5-7}
    & Group 3 &  &  & 5.0001e-04 & 7.0614e-04 & 5.6167e-04 \\ \cline{2-2} \cline{5-7}
Example \ref{ex5}& Group 4 &  &  & 1.3050e-03 & 1.8322e-03 & 1.4997e-03 \\ \cline{2-2} \cline{5-7}
    & Group 5 &  &  & 2.2948e-03 & 3.1820e-03 & 2.5513e-03 \\ \cline{2-2} \cline{5-7}
    & Group 6 &  &  & 3.8751e-03 & 5.4953e-03 & 4.4291e-03 \\ \cline{2-2} \cline{5-7}
    & Group 7 &  &  & 7.5156e-03 & 1.1594e-02 & 9.4983e-03 \\ \hline
\end{tabular}
\caption{The input parameters and the output statistics for MCX in Example \ref{ex3} - Example \ref{ex5}.}
\label{table:MCX}
\end{table}

\begin{example}[1D Slab Problem] \label{ex1}
In this example, we consider the benchmark problem of a one-dimensional critical slab with one-dimensional spatial and one-dimensional angular variables. The transport equation is
\begin{align*}
    \mu\frac{\partial \Psi(x,\mu)}{\partial x}+\Sigma_t(x)\Psi(x,\mu)=\frac{1}{2}\left(\Sigma_s(x)+\frac{\nu\Sigma_f(x)}{k_{\text{eff}}}\right)\int_{-1}^1\Psi(x,\mu')\mathrm{d}\mu',\quad\forall(x,\mu)\in(-b,b)\times(-1,1),
\end{align*}
with vacuum boundary conditions
\begin{align*}
    &\Psi(b,\mu)=0, \quad\forall\mu\in[-1,0],\\
    &\Psi(-b,\mu)=0, \quad\forall\mu\in[0,1].
\end{align*}
Material parameters are $\Sigma_t(x)=5\,m^{-1}$, $\Sigma_s(x)=3\,m^{-1}$, $\nu\Sigma_f(x)=2.25\,m^{-1}$, and slab thickness $b=0.6600527544\,m$. At criticality, $k_\text{eff}=1$. This example is treated as a fixed-eigenvalue benchmark at criticality, and the eigenfunction is normalized by prescribing two anchor values:
\begin{align*}
    \Psi(0,-1)=0.2, \quad\Psi(0,1)=0.2.
\end{align*}
\end{example}

This benchmark (\cite{Naz2007benchmark}) admits reference values at selected locations, which we use to validate the numerical results. We compare a finite difference method (FDM) with discrete ordinates and the RaNN method. We use the trapezoidal rule with $200$ quadrature points in the discrete ordinates method ($S_{200}$), and $M$ denotes the number of spatial mesh points. The values at the benchmark locations are obtained using linear interpolation. For the RaNN method, we set $N^I=50\times50$, $N^B=500+500$, and $r=10$. Numerical results are shown in Table \ref{table:ex1}. The value at $x=0$ is used as the normalization point, so the relative error there is identically zero for all methods.

\begin{table}[!htbp]
\centering
\begin{tabular}{|c|c|c|c|c|c|c|c|}
\hline
\multicolumn{2}{|c|}{$x/b$} & $0$ & $0.25$ & $0.5$ & $0.75$ & $1$ & time(s)\\ \hline\hline
\multicolumn{2}{|c|}{reference values (\cite{Naz2007benchmark})} & $1$ & $0.947144$ & $0.793726$ & $0.553290$ & $0.214192$ & \\ \hline\hline
\multirow{3}{*}{$S_{200}$} & $M=100$ & 0 & 2.5234e-03 & 7.1902e-03 & 3.3957e-03 & 5.5280e-03 & 0.685501 \\ \cline{2-8} 
    & $M=200$ & 0 & 1.3404e-03 & 2.1758e-04 & 1.3529e-03 & 1.6924e-03 & 1.353178 \\ \cline{2-8} 
    & $M=300$ & 0 & 3.3128e-04 & 3.7247e-03 & 4.1807e-04 & 8.0292e-04 & 1.967209 \\ \hline\hline
\multirow{3}{*}{RaNN} & $m=200$ & 0 & 2.5791e-04 & 1.2351e-03 & 9.9065e-04 & 9.4649e-02 & 0.017503 \\ \cline{2-8} 
    & $m=500$ & 0 & 2.3516e-05 & 5.2369e-05 & 6.4748e-05 & 4.5558e-02 & 0.041313 \\ \cline{2-8} 
    & $m=1000$ & 0 & 9.8406e-05 & 8.6062e-06 & 1.2036e-04 & 2.8078e-02 & 0.123724 \\ \hline
\end{tabular}
\caption{Pointwise relative errors and runtimes for the FDM and the RaNN method in Example \ref{ex1}. }
\label{table:ex1}
\end{table}

RaNN exhibits a clear accuracy--cost tradeoff as the number of random features increases. As $m$ increases from 200 to 1000, the pointwise relative errors decrease substantially, while the runtime increases only moderately and remains below that of the discrete-ordinates finite-difference baseline. These results indicate that, for this smooth one-dimensional transport benchmark, RaNN can achieve competitive accuracy with a relatively small number of degrees of freedom.

\begin{example}[2D Cylinder Problem] \label{ex2}
This example considers an infinite-height cylinder, involving one-dimensional spatial and two-dimensional angular variables. The transport equation is
\begin{align*}
    \sqrt{1-\mu^2}\left(\cos\varphi\frac{\partial}{\partial x}-\frac{\sin\varphi}{x}\frac{\partial}{\partial \varphi}\right)\Psi(x,\varphi,\mu)+\Sigma_t(x)\Psi(x,\varphi,\mu)=Q_s+Q_f,\quad\forall(x,\varphi,\mu)\in(0,R)\times(0,2\pi)\times(-1,1),
\end{align*}
where
\begin{align*}
    Q_s=\frac{1}{4\pi}\Sigma_s(x)\int_{-1}^1\int_0^{2\pi}\Psi(x,\varphi',\mu')\mathrm{d}\varphi'\mathrm{d}\mu',\quad
    Q_f=0.2\cos\left(\frac{\pi x}{2R}\right),\quad
    R=1.08225766\,m,
\end{align*}
and the boundary condition is vacuum:
\begin{align*}
    \Psi(R,\varphi,\mu)=0,\quad\forall(\varphi,\mu)\in\left[\frac{\pi}{2},\frac{3\pi}{2}\right]\times[-1,1].
\end{align*}
To avoid numerical issues near $x=0$ due to the $1/x$ term, we regularize the coefficient by replacing $x$ with $\max(x,10^{-4})$ in the implementation.

\noindent\textbf{Case 1.}
The material parameters are $\Sigma_t(x)=5\,m^{-1}$, $\Sigma_s(x)=3\,m^{-1}$.

\noindent\textbf{Case 2.}
The material parameters are
\begin{align*}
    \Sigma_t(x)=
    \begin{cases}
        5\,m^{-1},\quad 0<x<R/2\\0.5\,m^{-1},\quad R/2<x<R    
    \end{cases}
    \quad\text{and}\quad
    \Sigma_s(x)=
    \begin{cases}
        3\,m^{-1},\quad 0<x<R/2\\0.3\,m^{-1},\quad R/2<x<R.
    \end{cases}
\end{align*}
\end{example}

We compare the PINN method in \cite{Liu2025IDNT} with the RaNN method. The reference scalar flux is generated by MCX. The PINN results and computational times are taken directly from the reference, as the experiments were carried out under different hardware (Apple M1 Pro GPU) and software environments. Importantly, our comparison does not aim at a strict wall-clock-time benchmark under identical conditions. Instead, we emphasize a method-level comparison of computational efficiency, taking into account the fundamentally different training paradigms: iterative gradient-based optimization for PINNs versus linear least-squares solvers for RaNNs. This comparison is therefore intended only to provide a qualitative indication of the different computational characteristics of the two training paradigms.

\noindent\textbf{Case 1.}
We set $N^I=30\times30\times30$, $N^B=50\times50$, and $r=2$ for the RaNN method and vary the number of neurons $m$. The relative errors are presented in Table \ref{table:ex2_case1}. The PINN results correspond to the best-performing cases reported in \cite{Liu2025IDNT} (``PINN IDNT 4000'' and ``PINN EDNT 8000'' in Table 4 of \cite{Liu2025IDNT}). Numerical solutions and absolute differences are shown in Figure \ref{ex2_case1} for $m=800$.

\begin{table}[!htbp]
\centering
\begin{tabular}{|c||c|c|c|c||c|c|}
\hline
$m$ & 100 & 200 & 400 & 800 & PINN-IDNT & PINN-EDNT\\\hline
$e_{\ell^2}$ & 5.1349e-02 & 5.2151e-03 & 1.8812e-03 & 1.4699e-03 & 3.5101e-03 & 4.8811e-03 \\\hline 
time(s) & 0.0478 & 0.1049 & 0.2501 & 0.4822 & 28991.51 & 6525.75 \\\hline 
\end{tabular}
\caption{Relative $\ell^2$ errors of the scalar flux for RaNN and PINN in Example \ref{ex2} (Case 1). PINN results are quoted from \cite{Liu2025IDNT} and were obtained under different experimental conditions.}
\label{table:ex2_case1}
\end{table}

\begin{figure}[!htbp]
    \centering
    \includegraphics[width=0.4\textwidth]{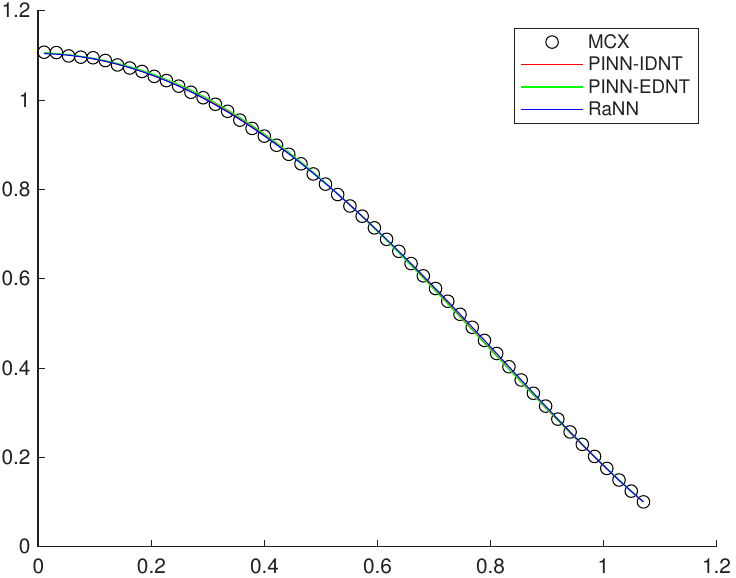}
    \includegraphics[width=0.4\textwidth]{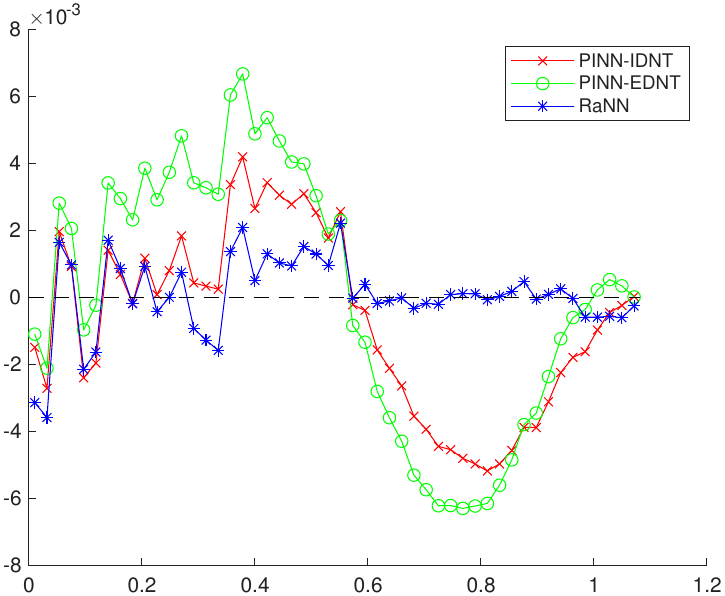}
    \caption{Scalar-flux profiles (left) and pointwise absolute differences from the MCX reference (right) for Example \ref{ex2} (Case 1). Results are shown for RaNN ($m=800$) and the PINN reported in \cite{Liu2025IDNT}.}
    \label{ex2_case1}
\end{figure}

\noindent\textbf{Case 2.}
This case involves two materials. We split the spatial domain into
\begin{align*}
D_1=(0,R/2)\times(0,2\pi)\times(-1,1),
\qquad
D_2=(R/2,R)\times(0,2\pi)\times(-1,1),
\end{align*}
and use two RaNNs $\Psi_\rho^1$ and $\Psi_\rho^2$. To couple the two subdomain networks, we enforce continuity of the angular flux on the interface
\begin{align*}
\Gamma_I=\{R/2\}\times[0,2\pi]\times[-1,1]
\end{align*}
through the weighted interface penalty
\begin{align*}
    \frac{|\Gamma_I|}{N^F}\sum_{i=1}^{N^F}
    \sqrt{1-\mu_i^2}\,|\cos\varphi_i|\,
    \big(\Psi_\rho^{1}(\bx_i^F)-\Psi_\rho^{2}(\bx_i^F)\big)^2,
\end{align*}
where $\bx_i^F=(R/2,\varphi_i,\mu_i)\in\Gamma_I$. Here the weight
\begin{align*}
    \omega_I=|\bn_I\cdot\bOmega|=\sqrt{1-\mu^2}\,|\cos\varphi|
\end{align*}
is exactly the transport trace weight on the interface, with $\bn_I$ the unit normal to $\Gamma_I$. We set $N^I=30\times30\times30$, $N^B=30\times30$, $N^F=30\times30$, $r_1=2$, $r_2=3$, and $m_1=m_2=2000$. Relative errors and runtime are reported in Table \ref{table:ex2_case2}, solutions and absolute differences are shown in Figure \ref{ex2_case2}.

\begin{table}[!htbp]
\centering
\begin{tabular}{|c||c|c|c|}
\hline
    & LRaNN & PINN-IDNT & PINN-EDNT\\\hline
$e_{\ell^2}$ & 1.5945e-02 & 3.8502e-02 & 3.7632e-02 \\\hline 
time(s) & 6.85 & 34706.51 & 4542.73 \\\hline 
\end{tabular}
\caption{Relative $\ell^2$ errors of LRaNN method and PINN method in Case 2 of Example \ref{ex2}. }
\label{table:ex2_case2}
\end{table}

\begin{figure}[!htbp]
    \centering
    \includegraphics[width=0.4\textwidth]{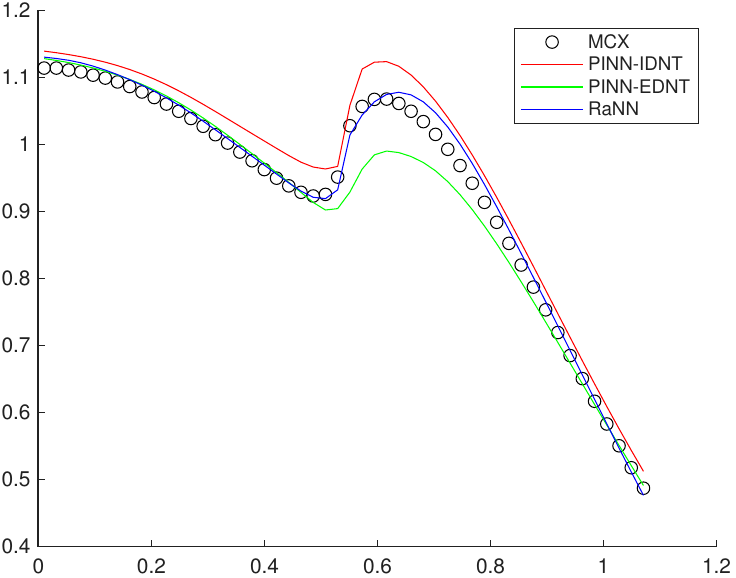}
    \includegraphics[width=0.4\textwidth]{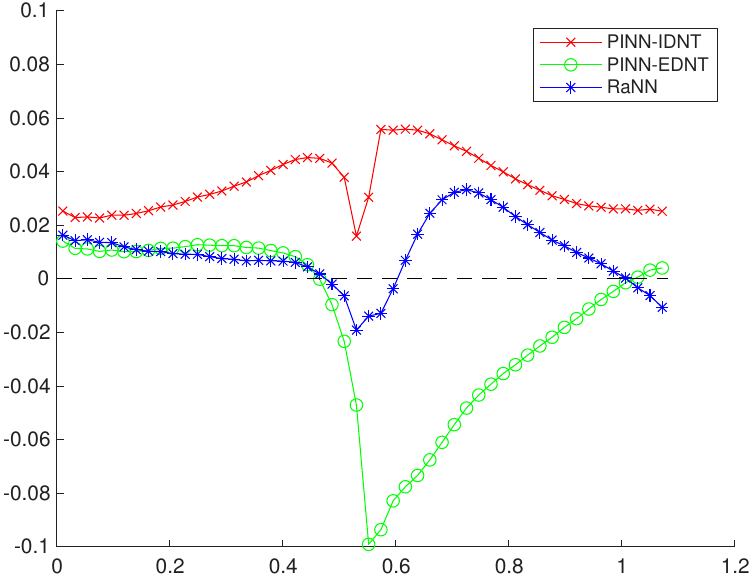}
    \caption{Scalar-flux profiles (left) and pointwise absolute differences from the MCX reference (right) for Example \ref{ex2} (Case 2). Results are shown for LRaNN and the PINN reported in \cite{Liu2025IDNT}.}
    \label{ex2_case2}
\end{figure}

RaNN achieves competitive accuracy with substantially reduced training complexity relative to gradient-based PINN methods. In the single-material case, increasing the number of neurons improves the scalar-flux accuracy rapidly, and moderate feature counts already produce relative $\ell^2$ errors comparable to the best reported PINN results, while requiring only a single linear least-squares solve. In the two-material case, the local RaNN formulation with an interface continuity penalty shows better performance for the reported heterogeneous setting and yields smaller errors than the reported PINN baselines. These observations suggest that fixing the nonlinear parameters and solving only for the linear output weights can reduce optimization sensitivity and provide an efficient alternative for transport models with nonlocal scattering.

\begin{figure}[!htbp]
    \centering
    \includegraphics[width=0.4\textwidth]{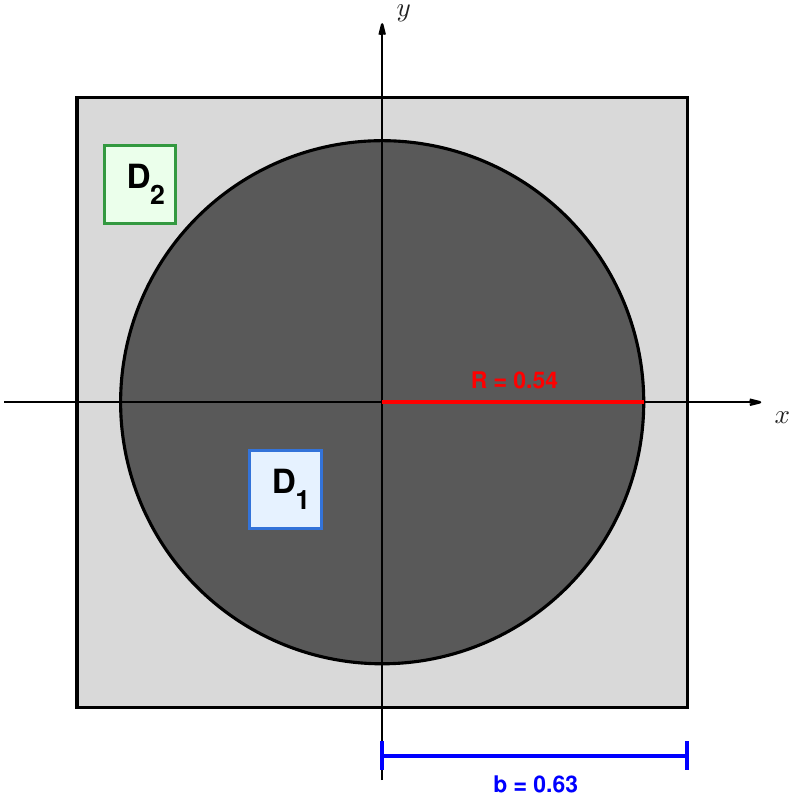}
    \includegraphics[width=0.4\textwidth]{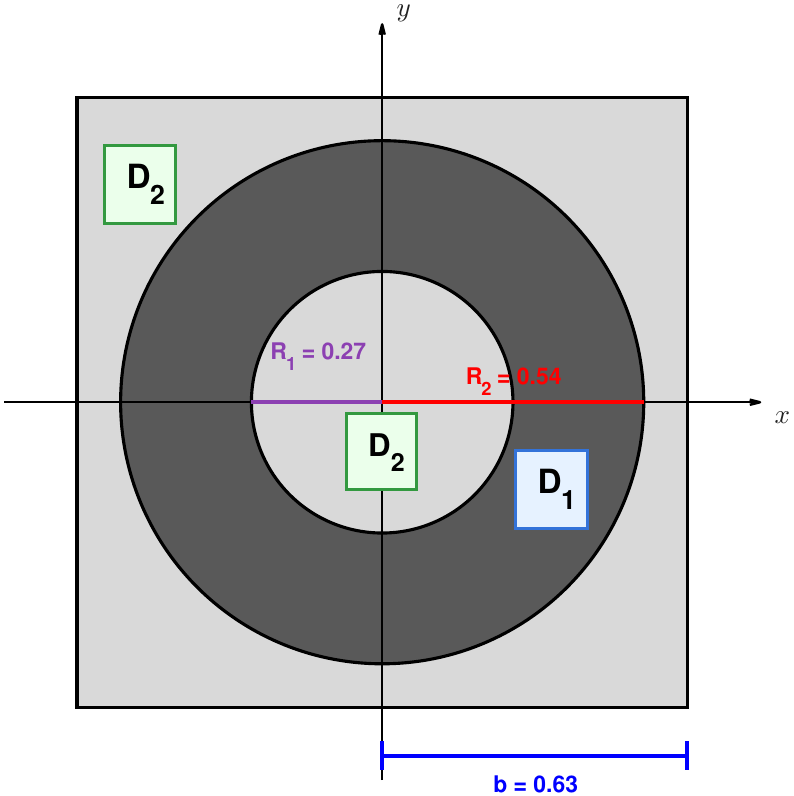}
    \caption{Schematic illustrations of the pin cells in Example \ref{ex3} and Example \ref{ex4}: Cases 1 and 2 are shown on the left, and Case 3 is shown on the right.}
    \label{fig:pincell}
\end{figure}

\begin{example}[Pin-cell Problem with Vacuum Boundary] \label{ex3}
We consider a two-dimensional pin-cell problem with two spatial variables $(x,y)$ and two angular variables $(\varphi,\mu)$. Figure \ref{fig:pincell} provides schematic illustrations of the pin-cell configurations. The steady-state transport equation is given by
    \begin{align*}
        \sqrt{1-\mu^2}\left(\cos\varphi\frac{\partial}{\partial x}+\sin\varphi\frac{\partial}{\partial y}\right)\Psi(x,y,\varphi,\mu)&+\Sigma_t(x,y)\Psi(x,y,\varphi,\mu) \\ 
        &=\frac{1}{4\pi}\Sigma_s(x,y)\int_{-1}^1 \int_0^{2\pi}\Psi(x,y,\varphi',\mu')\mathrm{d}\varphi'\mathrm{d}\mu' +\frac{1}{4\pi}Q(x,y),\nonumber
    \end{align*}
for all $(x,y,\varphi,\mu)\in(-b,b)^2\times(0,2\pi)\times(-1,1)$.
The boundary condition is a vacuum boundary condition, i.e.,
    \begin{align*}
        &\Psi(-b,y,\varphi,\mu)=0, \quad(y,\varphi,\mu)\in[-b,b]\times \left([0,\pi/2]\cup[3\pi/2,2\pi]\right)\times[-1,1],\\
        &\Psi(b,y,\varphi,\mu)=0, \quad(y,\varphi,\mu)\in[-b,b]\times[\pi/2,3\pi/2]\times[-1,1],\\
        &\Psi(x,-b,\varphi,\mu)=0, \quad(x,\varphi,\mu)\in[-b,b]\times[0,\pi]\times[-1,1],\\
        &\Psi(x,b,\varphi,\mu)=0, \quad(x,\varphi,\mu)\in[-b,b]\times[\pi,2\pi]\times[-1,1].
    \end{align*}
\noindent\textbf{Case 1.}
The pin-cell consists of two spatial subdomains,
\begin{align*}
    &D_{\br}^1=\{(x,y):x^2+y^2<R^2,R=0.54\,cm\},\\
    &D_{\br}^2=\{(x,y):(-b,b)^2\setminus\bar{D}_{\br}^1,b=0.63\,cm\}.
\end{align*}
The material parameters and the source term are 
\begin{align*}
    \Sigma_t(x,y)=1.25445\,cm^{-1},
    \quad
    \Sigma_s(x,y)=1.12\,cm^{-1},
    \quad
    Q(x,y)=
    \begin{cases}
        1,\quad(x,y)\in D_{\br}^1,\\
        0,\quad(x,y)\in D_{\br}^2.
    \end{cases}
\end{align*}
\noindent\textbf{Case 2.}
The pin-cell contains two parts:
\begin{align*}
    &D_{\br}^1=\{(x,y):x^2+y^2<R^2,R=0.54\,cm\},\\
    &D_{\br}^2=\{(x,y):(-b,b)^2\setminus\bar{D}_{\br}^1,b=0.63\,cm\}.
\end{align*}
The material parameters and the source term are 
\begin{align*}
&    \Sigma_t(x,y)=
    \begin{cases}
        0.395168\,cm^{-1},\quad(x,y)\in D_{\br}^1,\\
        1.25445\,cm^{-1},\quad(x,y)\in D_{\br}^2,
    \end{cases}
    \Sigma_s(x,y)=
    \begin{cases}
        0.265802\,cm^{-1},\quad(x,y)\in D_{\br}^1,\\
        1.12\,cm^{-1},\quad(x,y)\in D_{\br}^2,
    \end{cases}\\
&    Q(x,y)=
    \begin{cases}
        1,\quad(x,y)\in D_{\br}^1,\\
        0,\quad(x,y)\in D_{\br}^2.
    \end{cases}
\end{align*}
\noindent\textbf{Case 3.}
The pin-cell contains two parts:
\begin{align*}
    &D_{\br}^1=\{(x,y):R_1^2<x^2+y^2<R_2^2,R_1=0.27\,cm,R_2=0.54\,cm\},\\
    &D_{\br}^2=\{(x,y):(-b,b)^2\setminus\bar{D}_{\br}^1,b=0.63\,cm\}.
\end{align*}
The material parameters and the source term are 
\begin{align*}
&    \Sigma_t(x,y)=
    \begin{cases}
        0.395168\,cm^{-1},\quad(x,y)\in D_{\br}^1,\\
        1.25445\,cm^{-1},\quad(x,y)\in D_{\br}^2,
    \end{cases}
    \Sigma_s(x,y)=
    \begin{cases}
        0.265802\,cm^{-1},\quad(x,y)\in D_{\br}^1,\\
        1.12\,cm^{-1},\quad(x,y)\in D_{\br}^2,
    \end{cases}\\
&    Q(x,y)=
    \begin{cases}
        1,\quad(x,y)\in D_{\br}^1,\\
        0,\quad(x,y)\in D_{\br}^2.
    \end{cases}
\end{align*}
\end{example}

\begin{table}[!htbp]
\centering
\begin{tabular}{|c|c||c|c|c|c|c|c|}
\hline
    \multicolumn{2}{|c||}{} & Case 1 & Case 2 & Case 3 \\\hline
    \multicolumn{2}{|c||}{$r$} & 2 & 2 & 2 \\\hline
    \multirow{2}{*}{RaNN} & $e_{\ell^2}$ & 2.1970e-02 & 2.2756e-02 & 6.5209e-02 \\\cline{2-5}
    & time(s) & 682.04 & 686.78 & 655.53 \\\hline
    \multirow{2}{*}{RaNN-S} & $e_{\ell^2}$ & 2.5098e-02 & 2.4902e-02 & 6.7096e-02 \\\cline{2-5}
    & time(s) & 392.22 & 369.59 & 342.40 \\\hline
\end{tabular}
\caption{Parameter settings and relative $\ell^2$ errors of the scalar flux for RaNN in Example \ref{ex3} (RaNN-S: sketching acceleration with the same settings as RaNN).}
\label{table:ex3}
\end{table}

\begin{figure}[!htbp]
    \centering
    \includegraphics[width=\textwidth]{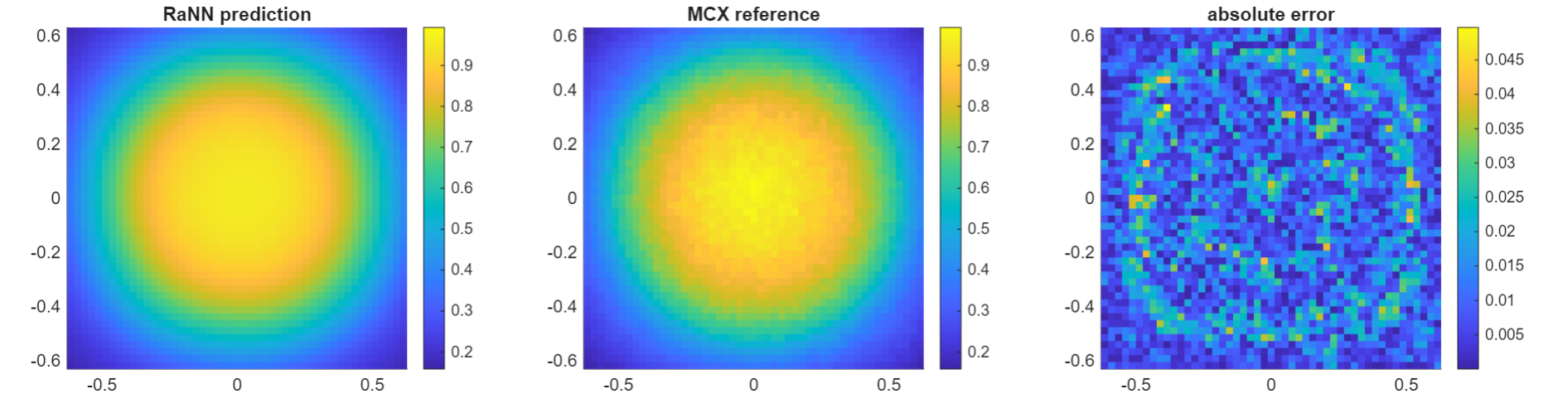}
    \caption{Scalar flux for Example \ref{ex3} (Case 1): RaNN prediction (left), MCX reference (middle), and absolute error (right).}
    \label{ex3-vacuumboundary-1region-1net}
\end{figure}

\begin{figure}[!htbp]
    \centering
    \includegraphics[width=\textwidth]{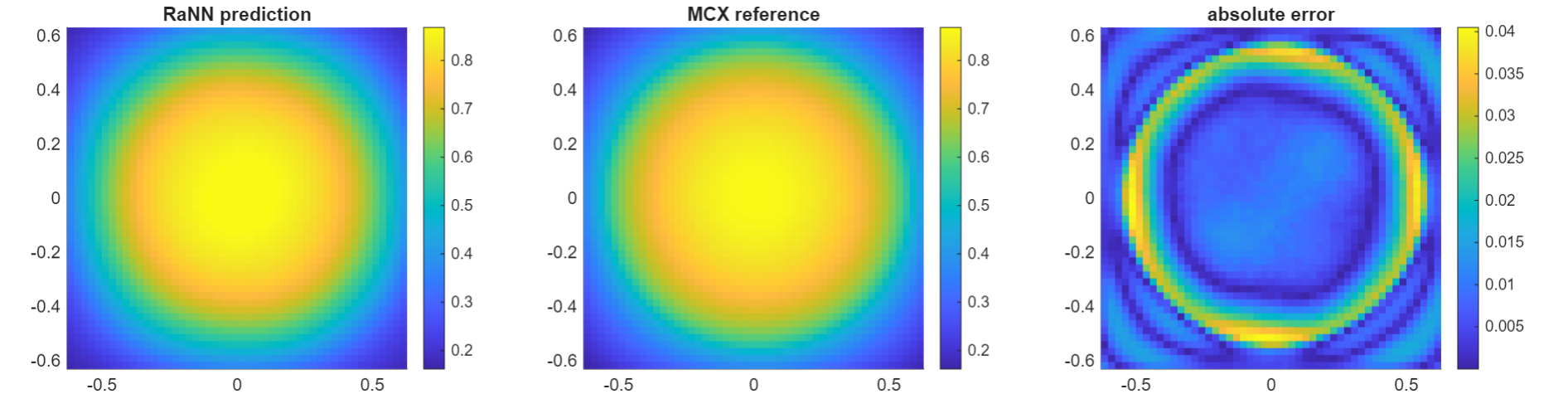}
    \caption{Scalar flux for Example \ref{ex3} (Case 2): RaNN prediction (left), MCX reference (middle), and absolute error (right).}
    \label{ex3-vacuumboundary-2region-1net}
\end{figure}

\begin{figure}[!htbp]
    \centering
    \includegraphics[width=\textwidth]{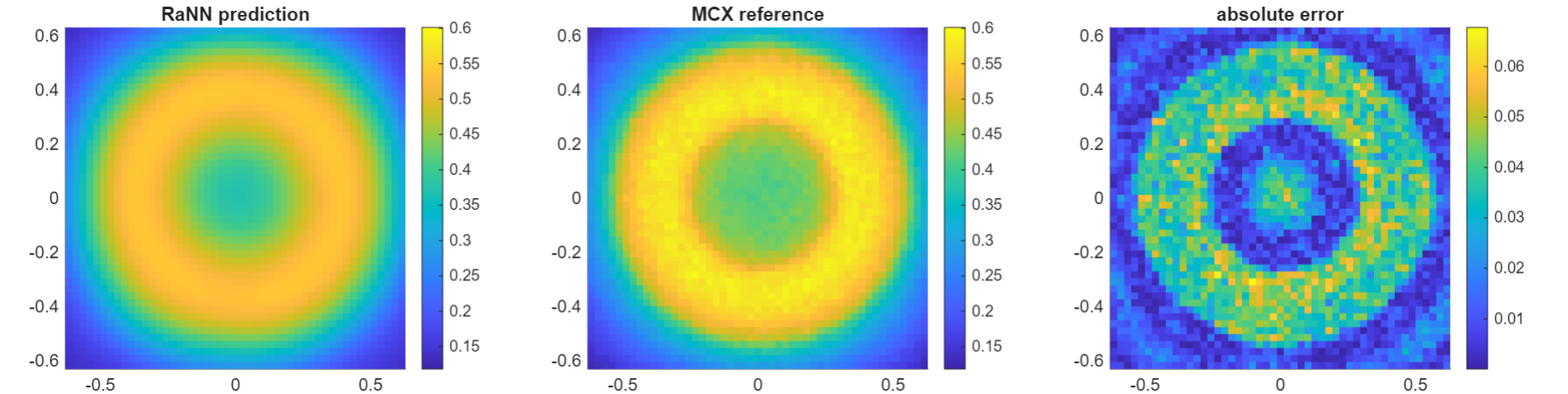}
    \caption{Scalar flux for Example \ref{ex3} (Case 3): RaNN prediction (left), MCX reference (middle), and absolute error (right).}
    \label{ex3-vacuumboundary-3region-1net}
\end{figure}

In this example, we have the same setting: $N^I=31\times31\times31\times31$, $N^B=4\times30\times30\times30$ and $m=8000$. In the RaNN method, the parameters and relative $\ell^2$ errors are shown in Table \ref{table:ex3}. For the different cases considered, the numerical solutions, reference solutions, and absolute errors (without sketching acceleration) are presented in Figures \ref{ex3-vacuumboundary-1region-1net}, \ref{ex3-vacuumboundary-2region-1net}, and \ref{ex3-vacuumboundary-3region-1net}.

The RaNN method remains stable at large collocation sizes but becomes dominated by the cost of solving the resulting large least-squares system, reflecting the high-dimensional phase-space discretization. Across the three material configurations, the method attains moderate accuracy, with the annular-fuel configuration (Case 3) being the most challenging and producing the largest error, consistent with its stronger spatial heterogeneity and sharper solution features. Sketching substantially reduces the computational time while preserving essentially the same accuracy, demonstrating that randomized compression is effective for accelerating the linear algebra stage when the augmented system has many more rows than columns. This example indicates that scalable implementations of RaNN for high-dimensional transport benefit strongly from randomized numerical linear algebra techniques.

\begin{example}[Pin-cell Problem with Reflecting Boundary] \label{ex4}
    In this example, we consider the same three material configurations as in Example \ref{ex3}, but impose reflecting boundary conditions. In the present setting, the reflection map on each side of the square domain is given by
    \begin{align*}
        &\Psi(-b,y,\varphi,\mu)=\Psi(-b,y,\pi-\varphi,\mu), \quad(y,\varphi,\mu)\in[-b,b]\times[0,\pi/2]\times[-1,1],\\
        &\Psi(-b,y,\varphi,\mu)=\Psi(-b,y,3\pi-\varphi,\mu), \quad(y,\varphi,\mu)\in[-b,b]\times[3\pi/2,2\pi]\times[-1,1],\\
        &\Psi(b,y,\varphi,\mu)=\Psi(b,y,\pi-\varphi,\mu), \quad(y,\varphi,\mu)\in[-b,b]\times[\pi/2,\pi]\times[-1,1],\\
        &\Psi(b,y,\varphi,\mu)=\Psi(b,y,3\pi-\varphi,\mu), \quad(y,\varphi,\mu)\in[-b,b]\times[\pi,3\pi/2]\times[-1,1],\\
        &\Psi(x,-b,\varphi,\mu)=\Psi(x,-b,2\pi-\varphi,\mu), \quad(x,\varphi,\mu)\in[-b,b]\times[0,\pi]\times[-1,1],\\
        &\Psi(x,b,\varphi,\mu)=\Psi(x,b,2\pi-\varphi,\mu), \quad(x,\varphi,\mu)\in[-b,b]\times[\pi,2\pi]\times[-1,1].
    \end{align*}
\end{example}

We use the same collocation setting as in Example \ref{ex3}, but with different random distribution parameters $r$. The corresponding parameter choices and relative $\ell^2$ errors of the scalar flux are reported in Table \ref{table:ex4}. For each case, the RaNN prediction, the MCX reference solution, and the absolute error (without sketching acceleration) are shown in Figures \ref{ex4-reflectionboundary-1region-1net}--\ref{ex4-reflectionboundary-3region-1net}.

\begin{table}[!htbp]
\centering
\begin{tabular}{|c|c||c|c|c|c|c|c|}
\hline
    \multicolumn{2}{|c||}{} & Case 1 & Case 2 & Case 3 \\\hline
    \multirow{3}{*}{RaNN} & $r$ & 4 & 2 & 3 \\\cline{2-5}
    & $e_{\ell^2}$ & 4.0388e-03 & 4.5391e-03 & 8.0339e-03 \\\cline{2-5}
    & time(s) & 685.16 & 686.21 & 664.10 \\ \hline
    \multirow{2}{*}{RaNN-S} & $e_{\ell^2}$ & 4.1324e-03 & 8.0447e-03 & 7.8753e-03 \\\cline{2-5}
    & time(s) & 347.07 & 355.65 & 341.70 \\ \hline
    \multirow{3}{*}{FV} & $(N_\varphi,N_\mu)$ & \multicolumn{3}{c|}{(16,16)} \\\cline{2-5}
     & $e_{\ell^2}$ & 5.8339e-03 & 7.8757e-03 & 1.5400e-02 \\\cline{2-5}
     & time(s) & 534.18 & 656.96 & 542.87 \\ \hline
\end{tabular}
\caption{Parameter settings and relative $\ell^2$ errors of the scalar flux for RaNN and FV in Example \ref{ex4}. }
\label{table:ex4}
\end{table}

\begin{figure}[!htbp]
    \centering
    \includegraphics[width=\textwidth]{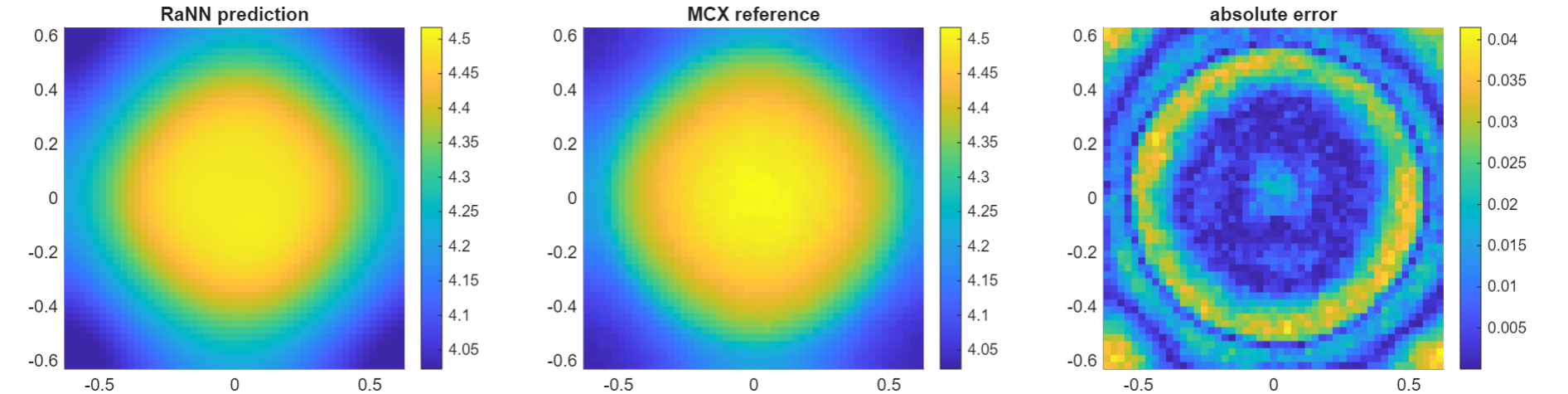}
    \caption{Scalar flux for Example \ref{ex4} (Case 1): RaNN prediction (left), MCX reference (middle), and absolute error (right).}
    \label{ex4-reflectionboundary-1region-1net}
\end{figure}

\begin{figure}[!htbp]
    \centering
    \includegraphics[width=\textwidth]{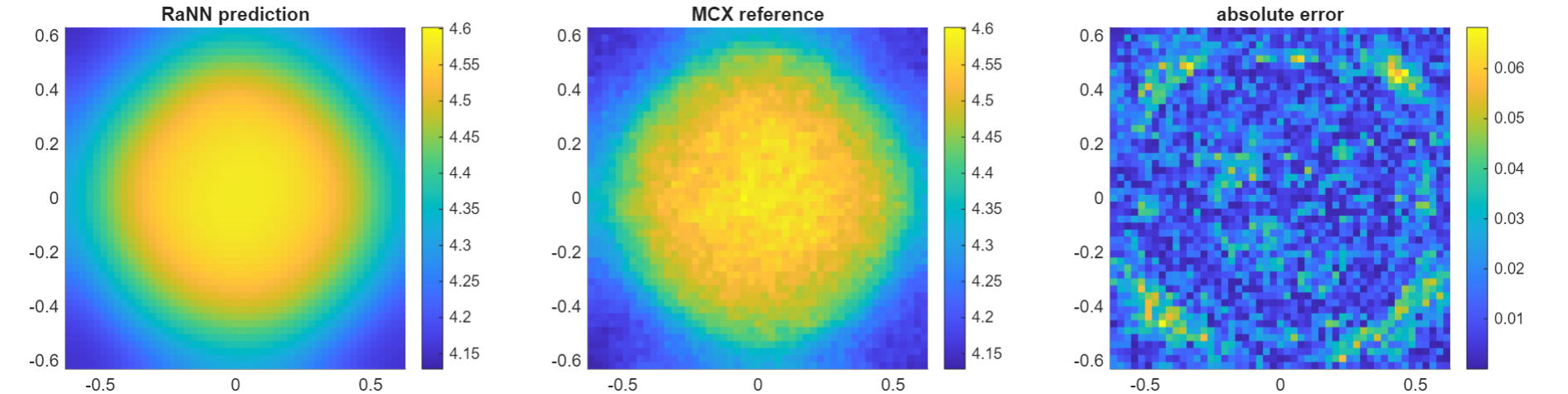}
    \caption{Scalar flux for Example \ref{ex4} (Case 2): RaNN prediction (left), MCX reference (middle), and absolute error (right).}
    \label{ex4-reflectionboundary-2region-1net}
\end{figure}

\begin{figure}[!htbp]
    \centering
    \includegraphics[width=\textwidth]{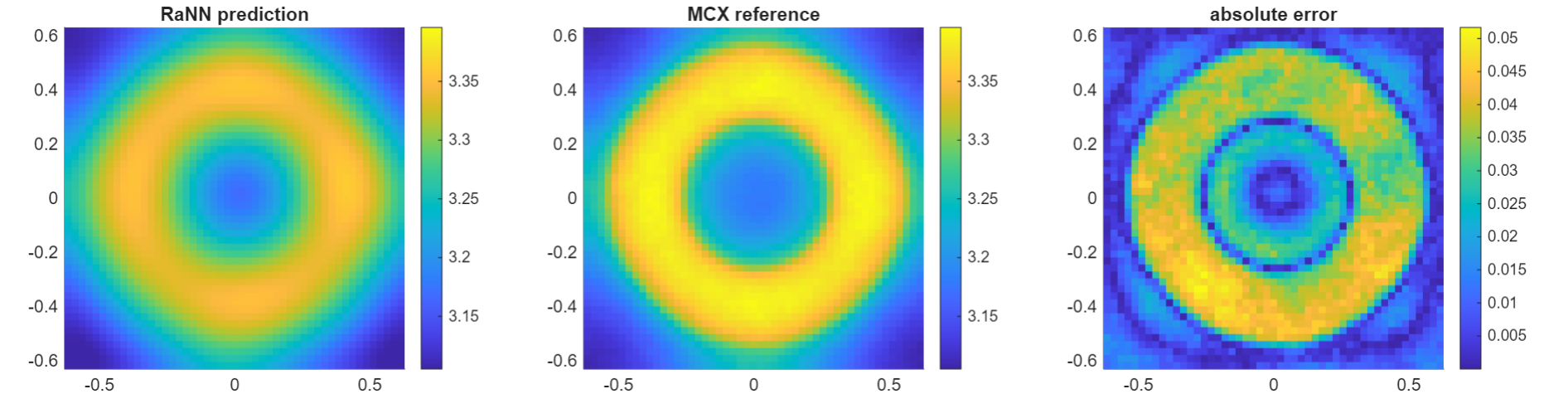}
    \caption{Scalar flux for Example \ref{ex4} (Case 3): RaNN prediction (left), MCX reference (middle), and absolute error (right).}
    \label{ex4-reflectionboundary-3region-1net}
\end{figure}

For comparison, we also solve the same transport problems using an $S_N$ discrete-ordinates formulation coupled with a finite-volume (FV) discretization in space. For the angular discretization, the $\varphi$-integral is approximated by a composite trapezoidal rule with $N_\varphi$ nodes, while the $\mu$-integral is evaluated using an $N_\mu$ point Gauss-Legendre quadrature. In space, we use a $50\times 50$ uniform grid. As in the previous tests, the MCX solution is taken as the reference. When $(N_\varphi,N_\mu)=(16,16)$, the relative $\ell^2$ errors are summarized in Table \ref{table:ex4}, and the corresponding absolute error fields are shown in Figure \ref{ex4-reflectionboundary-sn-16-16}. 

In particular, when $(N_\varphi,N_\mu)$ is increased to $(32,32)$, the resulting linear system exceeds the available memory on our platform. These results suggest that the RaNN formulation can remain competitive in regimes where high-resolution $S_N$ discretizations become memory intensive on the tested platform.

\begin{figure}[!htbp]
    \centering
    \includegraphics[width=\textwidth]{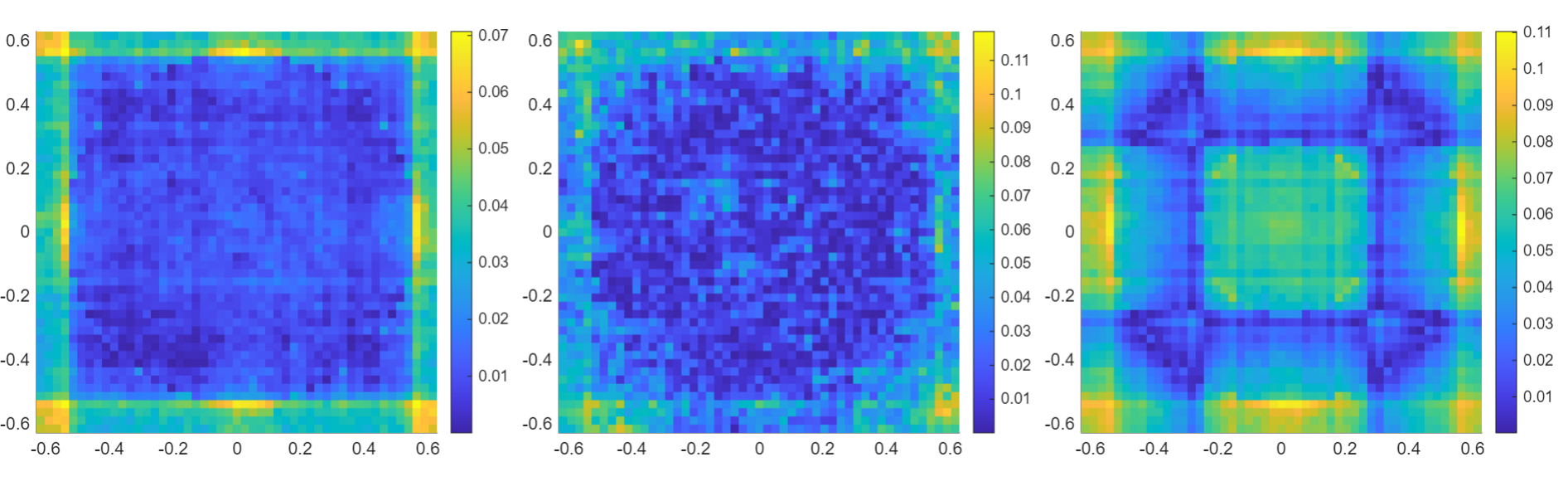}
    \caption{Absolute errors for FV as $(N_\varphi,N_\mu)=(16,16)$ in Example \ref{ex4}, left to right: Case 1 - Case 3.}
    \label{ex4-reflectionboundary-sn-16-16}
\end{figure}

In the reported test cases, the reflecting boundary problems lead to smaller relative errors than the corresponding vacuum cases. The sketching variant again provides a significant runtime reduction with only a modest degradation in accuracy, illustrating a practical trade-off between computational cost and accuracy for large least-squares systems. In comparison, the finite-volume discrete-ordinates (FV--$S_N$) solver attains similar error levels at moderate angular resolution, but memory requirements grow quickly with $(N_\varphi,N_\mu)$, limiting feasible angular refinement on the tested platform. Collectively, these observations emphasize that RaNN can deliver competitive accuracy without explicitly forming the large coupled systems typical of high-resolution $S_N$ discretizations, particularly in settings where angular refinement becomes memory intensive.

\begin{example}[7-Groups Pin-cell Problem] \label{ex5}
We consider a two-dimensional, seven-group pin-cell problem. The governing transport equations take the form
\begin{align*}
    \sqrt{1-\mu^2}\left(\cos\varphi\frac{\partial}{\partial x}+\sin\varphi\frac{\partial}{\partial y}\right)&\Psi_g(x,y,\varphi,\mu)+\Sigma_{t,g}(x,y)\Psi_g(x,y,\varphi,\mu)\\
    &=\frac{1}{4\pi}\sum_{g'=1}^7\Sigma_{s,g'\to g}(x,y)\int_{-1}^1 \int_0^{2\pi}\Psi_{g'}(x,y,\varphi',\mu')\mathrm{d}\varphi'\mathrm{d}\mu'+\frac{1}{4\pi}Q_g(x,y),
\end{align*}
for $g=1,\cdots,7$ and $(x,y,\varphi,\mu)\in(-b,b)^2\times(0,2\pi)\times(-1,1)$, where $b=0.63\,cm$. The scattering term couples different energy groups through the macroscopic cross sections $\Sigma_{s,g'\to g}$, so that the scattering source term for group $g$ depends on the angular fluxes in groups $g'$. Reflecting boundary conditions are imposed. The pin-cell consists of two spatial subdomains,
\begin{align*}
    &D_{\br}^1=\{(x,y):x^2+y^2<R^2,R=0.54\,cm\},\\
    &D_{\br}^2=\{(x,y):(-b,b)^2\setminus\bar{D}_{\br}^1\}.
\end{align*}
The macroscopic group constants are listed in Table \ref{table:ex5_groupconst1}--Table \ref{table:ex5_groupconst3}.

\begin{table}[!htbp]
\centering
\begin{tabular}{|c|c|c|c|c|c|c|c|}
\hline
$g$ & Group 1 & Group 2 & Group 3 & Group 4 & Group 5 & Group 6 & Group 7 \\\hline
$D_{\br}^1$ & 3.558980e-1 & 6.596100e-1 & 9.607760e-1 & 1.108734e-0 & 6.236020e-1 & 7.903360e-1 & 1.128812e-0 \\\hline 
$D_{\br}^2$ & 3.184120e-1 & 8.259400e-1 & 1.180620e-0 & 1.168700e-0 & 1.436000e-0 & 2.508900e-0 & 5.300760e-0 \\\hline 
\end{tabular}
\caption{Macroscopic total cross sections $\Sigma_{t,g}$ (cm$^{-1}$) for the two regions in Example \ref{ex5}. }
\label{table:ex5_groupconst1}
\end{table}

\begin{table}[!htbp]
\centering
\begin{tabular}{|c|c|c|c|c|c|c|c|}
\hline
\diagbox{$g'$}{$\Sigma_{s,g'\to g}$}{$g$} & to Group 1 & to Group 2 & to Group 3 & to Group 4 & to Group 5 & to Group 6 & to Group 7 \\\hline
Group 1 & 1.27537e-1 & 4.37800e-2 & 9.43740e-6 & 5.51630e-9 & 0 & 0 & 0 \\\hline 
Group 2 & 0 & 3.24456e-1 & 1.63140e-3 & 3.14270e-9 & 0 & 0 & 0 \\\hline 
Group 3 & 0 & 0 & 4.50940e-1 & 2.67920e-3 & 0 & 0 & 0 \\\hline 
Group 4 & 0 & 0 & 0 & 4.52565e-1 & 5.56640e-3 & 0 & 0 \\\hline 
Group 5 & 0 & 0 & 0 & 1.25250e-4 & 2.71401e-1 & 1.02550e-2 & 1.00210e-8 \\\hline 
Group 6 & 0 & 0 & 0 & 0 & 1.29680e-3 & 2.65802e-1 & 1.68090e-2 \\\hline 
Group 7 & 0 & 0 & 0 & 0 & 0 & 8.54580e-3 & 2.73080e-1 \\\hline 
\end{tabular}
\caption{Macroscopic scattering cross sections $\Sigma_{s,g'\to g}$ (cm$^{-1}$) in region $D_{\br}^1$ for Example \ref{ex5}. }
\label{table:ex5_groupconst2}
\end{table}

\begin{table}[!htbp]
\centering
\begin{tabular}{|c|c|c|c|c|c|c|c|}
\hline
\diagbox{$g'$}{$\Sigma_{s,g'\to g}$}{$g$} & to Group 1 & to Group 2 & to Group 3 & to Group 4 & to Group 5 & to Group 6 & to Group 7 \\\hline
Group 1 & 4.44777e-2 & 1.13400e-1 & 7.23470e-4 & 3.74990e-6 & 5.31840e-8 & 0 & 0 \\\hline 
Group 2 & 0 & 2.82334e-1 & 1.29940e-1 & 6.23400e-4 & 4.80020e-5 & 7.44860e-6 & 1.04550e-6 \\\hline 
Group 3 & 0 & 0 & 3.45256e-1 & 2.24570e-1 & 1.69990e-2 & 2.64430e-3 & 5.03440e-4 \\\hline 
Group 4 & 0 & 0 & 0 & 9.10284e-2 & 4.15510e-1 & 6.37320e-2 & 1.21390e-2 \\\hline 
Group 5 & 0 & 0 & 0 & 7.14370e-5 & 1.39138e-1 & 5.11820e-1 & 6.12290e-2 \\\hline 
Group 6 & 0 & 0 & 0 & 0 & 2.21570e-3 & 6.99913e-1 & 5.37320e-1 \\\hline 
Group 7 & 0 & 0 & 0 & 0 & 0 & 1.32440e-1 & 2.48070e-0 \\\hline 
\end{tabular}
\caption{Macroscopic scattering cross sections $\Sigma_{s,g'\to g}$ (cm$^{-1}$) in region $D_{\br}^2$ for Example \ref{ex5}. }
\label{table:ex5_groupconst3}
\end{table}
\end{example}

In this example, the sketching technique described in Section \ref{sec:sketching} is employed, and the same sketching matrix $S$ is used for all least-squares solves. Based on the sparsity patterns of the group-to-group scattering matrices in Tables \ref{table:ex5_groupconst2}--\ref{table:ex5_groupconst3}, the multigroup system exhibits a block structure: Groups 1--3 form a lower-triangular subsystem that can be solved sequentially, whereas Groups 4--7 remain mutually coupled and are therefore solved simultaneously through a block least-squares problem.

We set $N^I=31\times31\times31\times31$, $N^B=4\times30\times30\times30$, $m=6000$ and $r=2$. The relative $\ell^2$ errors for different groups and runtime are shown in Table \ref{table:ex5}. The numerical solutions, reference solutions, and absolute errors are presented in Figure \ref{ex5-reflectionboundary-2region-1net}.

Note that we set $m=6000$ due to computational constraints (using more neurons would lead to an out-of-memory error), though a larger number of neurons would yield better results. The reported errors increase with group index, indicating that accurately resolving the down-scattered, strongly coupled low-energy behavior requires greater representational capacity. This example suggests that further improvements will likely require larger feature budgets (or adaptive/localized bases), more efficient memory management, and enhanced multigroup coupling strategies to scale RaNN reliably to high-fidelity multigroup transport.

\begin{table}[!htbp]
\centering
\resizebox{\linewidth}{!}{ 
\begin{tabular}{|c|c|c|c|c|c|c|c||c|}
\hline
$g$ & Group 1 & Group 2 & Group 3 & Group 4 & Group 5 & Group 6 & Group 7 & time(s) \\\hline
$e_{\ell^2}$ & 5.6894e-02 & 3.1044e-02 & 4.5269e-02 & 9.8361e-02 & 1.4397e-01 & 1.9357e-01 & 2.4744e-01 & 3942.39 \\\hline 
\end{tabular}
}
\caption{Relative $\ell^2$ errors and runtime of RaNN-S method in Example \ref{ex5}. }
\label{table:ex5}
\end{table}

\begin{figure}[!htbp]
    \centering
    \includegraphics[width=\textwidth]{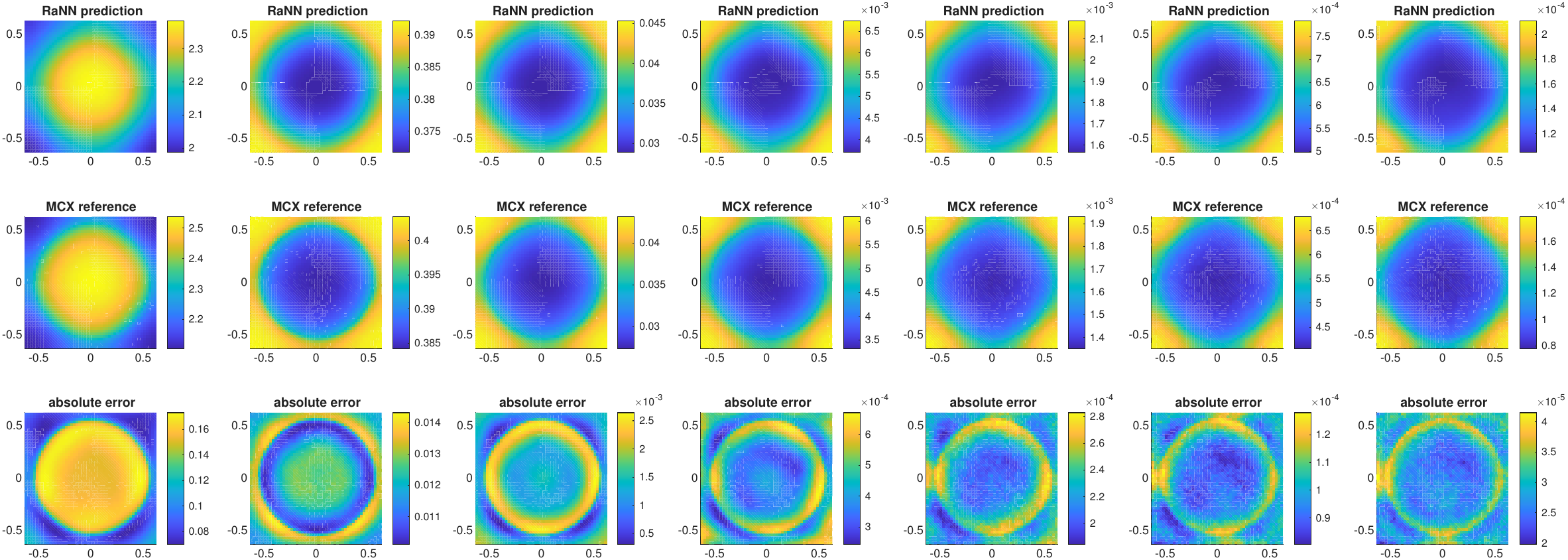}
        \caption{For each energy group (left to right: Groups 1-7), we show the RaNN scalar flux (top), the MCX reference (middle), and the absolute error (bottom).}
    \label{ex5-reflectionboundary-2region-1net}
\end{figure}

\section{Summary} \label{sec:summary}

This work presents randomized neural networks as a general solver for linear integro-differential equations and demonstrates the framework on the steady neutron transport equation as a representative high-dimensional nonlocal model. The central idea is to randomly fix the hidden-layer parameters and train only the linear output weights, thereby reducing neural training to a linear least-squares problem in the output coefficients. This formulation yields a global minimizer of the empirical least-squares objective in the output coefficients and significantly reduces training cost. As a result, the proposed approach provides a practical way to address two common challenges in nonlocal models: (i) dense coupling induced by integral operators in deterministic discretizations, and (ii) the nonconvex optimization and high computational expense often encountered in PINN-type methods.

We have detailed the RaNN formulation for the neutron transport case study, including exact evaluation of differential operators via analytical differentiation, incorporation of scattering integrals through numerical quadrature, and extension to multi-material configurations using local RaNNs coupled by interface conditions. Numerical benchmarks covering a one-dimensional slab, a two-dimensional infinite cylinder, two-dimensional slabs with vacuum and reflecting boundaries, and a challenging two-dimensional seven-group problem with sketching acceleration show that RaNNs achieve competitive accuracy while offering reduced computational cost relative to the selected classical and neural baselines in the reported settings. These results indicate that RaNNs provide a robust, mesh-free, and efficient alternative for high-fidelity simulation of linear transport models with nonlocal integral operators.

Future work will focus on extending the RaNN framework to more complex geometries and larger-scale problems, including fully three-dimensional domains, irregular geometries, and substantially increased numbers of energy groups. Addressing these challenges will require efficient preprocessing for domain decomposition, improved strategies for strongly coupled multigroup systems, and the integration of high-performance computing techniques, such as GPU acceleration, to accommodate growing computational demands.

\end{document}